\documentclass[a4paper,12pt]{amsart}
\usepackage{amsmath,amssymb,amsthm,leftidx}

\newcommand{\supp}{\operatorname{supp}}

\renewcommand{\>}{\rangle}
\newcommand{\<}{\langle}

\newcommand{\IM}{\operatorname{\mathcal{M}}}

\newcommand{\EE}{\operatorname{\mathcal{E}}}
\newcommand{\LL}{\operatorname{\mathcal{L}}}

\renewcommand{\AA}{\operatorname{\mathfrak{A}}}
\newcommand{\PP}{\operatorname{\mathfrak{P}}}

\newcommand{\IS}{\operatorname{\mathcal{S}}}

\newcommand{\CP}{\mathbb{C}\mathbb{P}}
\newcommand{\eps}{\epsilon}

\newcommand{\CR}{\bar{\partial}}

\newcommand{\del}{\partial}

\newcommand{\IC}{\operatorname{\mathbb{C}}}
\newcommand{\IZ}{\operatorname{\mathbb{Z}}}

\newcommand{\IR}{\operatorname{\mathbb{R}}}
\newcommand{\IN}{\operatorname{\mathbb{N}}}
\newcommand{\IH}{\operatorname{\mathbb{H}}}
\newcommand{\IF}{\operatorname{\mathbb{F}}}
\newcommand{\IP}{\operatorname{\mathbb{P}}}
\renewcommand{\IS}{\operatorname{\mathbb{S}}}

\renewcommand{\LL}{\operatorname{\mathcal{L}}}
\renewcommand{\AA}{\operatorname{\mathcal{A}}}

\newtheorem{theorem}{Theorem}[section]
\newtheorem{proposition}[theorem]{Proposition}

\newtheorem{definition}[theorem]{Definition}
\newtheorem{lemma}[theorem]{Lemma}
\newtheorem{corollary}[theorem]{Corollary}
\newtheorem{remark}[theorem]{Remark}

\title{Floer theory for Hamiltonian PDE\\ using model theory}
\author{Oliver Fabert}
\thanks{O. Fabert, VU Amsterdam, The Netherlands. Email: o.fabert@vu.nl}
\begin{document}
\maketitle

\begin{abstract}
Under natural restrictions it is known that a nonlinear Schr\"odinger equation is a Hamiltonian PDE which defines a symplectic flow on a symplectic Hilbert space preserving the Hilbert norm. When the potential is one-periodic in time and after passing to the projectivization, it makes sense to ask whether the natural analogue of the Arnold conjecture holds. By employing methods from non-standard model theory we show how Hamiltonian Floer theory can be generalized from finite to infinite dimensions. While our proof entirely builds on finite-dimensional results, we do not ask for any prior knowledge of non-standard model theory.   
\end{abstract}

\tableofcontents
\markboth{O. Fabert}{Floer theory for Hamiltonian PDE} 

\section{Hamiltonian partial differential equations}

Nonlinear Schr\"odinger equations play a very important role in mathematical physics and have applications in, e.g., solid
state physics, condensed matter physics, quantum chemistry, nonlinear optics, wave propagation, protein folding and the semiconductor industry. In contrast to the well-known linear Schr\"odinger equation describing the time evolution of the quantum wave function of a single particle, nonlinear Schr\"odinger equations are classical field equations describing multi-particle systems, where the nonlinearity models the interaction between different particles. An example of a nonlinear Schr\"odinger equation is the so-called Gross-Pitaevskii equation $$i\del_t u \,=\, - \Delta u \;+\; c|u|^2 u\;+\; V(t,x)u,$$ which plays an important role in the theory of Bose-Einstein condensates. Here $u=u(t,x)\in\IC$ is a complex-valued function depending on time and space, $\del_t$ is the derivative with respect to the time $t\in\IR$, $\Delta$ denotes the Laplace operator with respect to the space coordinate $x$,  $V(t,x)$ is a time-dependent exterior potential and $c\in\IR$ is a scalar whose sign depends on whether the particles are attracting or repelling each other. Here and in what follows we restrict ourself to the case of one spatial dimension for notational simplicity; we claim that everything, including our main theorem, can be generalized to the higher-dimensional case. \\

Nonlinear Schr\"odinger equations are important examples of \emph{Hamiltonian partial differential equations}, where we refer to \cite{K} for definitions, statements and further references. This means that they can be written in the form $\del_t u=X^H_t(u)$, where the Hamiltonian vector field $X^H_t$ is determined by the choice of a (time-dependent) Hamiltonian function $H=H_t$ and a linear symplectic form $\omega$. Here a bilinear form $\omega:\IH\times\IH\to\IR$ on a real Hilbert space $\IH$ is called symplectic if it is anti-symmetric and nondegenerate in the sense that the induced linear mapping $i_{\omega}:\IH\to\IH^*$ is an isomorphism. As in the finite-dimensional case it can be shown that for any symplectic form $\omega$ there exists a complex structure $J_0$ on $\IH$ such that $\omega$, $J_0$ and the real inner product $\langle\cdot,\cdot\rangle_{\IR}$ on $\IH$ are related via $\langle\cdot,\cdot\rangle_{\IR}=\omega(\cdot,J_0\cdot)$.  \\

In the case of nonlinear Schr\"odinger equations on the circle $S^1=\IR/2\pi\IZ$ one chooses the complex Hilbert space $\IH=L^2(S^1,\IC)$ of square-integrable complex-valued functions on the circle which naturally can be viewed as a real Hilbert space by identifying $\IC$ with $\IR^2$. The standard complex inner product $\<\cdot,\cdot\>_{\IC}$ is related to the standard real inner product $\<\cdot,\cdot\>_{\IR}$ and the standard symplectic form $\omega$ by $\<\cdot,\cdot\>_{\IC}=\<\cdot,\cdot\>_{\IR}+i\omega$ and the symplectic form is related to the real inner product via $\omega=\<J_0\cdot,\cdot\>_{\IR}$ with $J_0=i$ denoting the standard complex structure on $\IH$. In order to stress the relation with the finite-dimensional case of $\IR^{2n}=\IC^n$, note that, using the Fourier series expansion $u(x)=(2\pi)^{-1/2}\sum_{n=-\infty}^{\infty} \hat{u}(n)\cdot \exp(inx)$ and writing $\hat{u}(n)=q_n+ip_n$ for all $n\in\IZ$, it follows that the symplectic Hilbert space $L^2(S^1,\IC)$ can be identified with the space $\ell^2(\IC)$ of square-summable complex-valued series $\hat{u}:\IZ\to\IC$ equipped with the symplectic form $\omega=\sum_{n=-\infty}^{+\infty} dp_n\wedge dq_n$. The corresponding Hamiltonian function is of the form $$H_t(u)\,=\, \int_0^{2\pi} \frac{|u_x(x)|^2}{2}\;dx \;+\; F_t(u)$$ with $$F_t(u)\,=\, \int_0^{2\pi} \frac{1}{2} f(|u(x)|^2,x,t)\,dx,$$ where $f$ is a smooth, real-valued function on $\IR^+\times S^1\times\IR$. Note that the Gross-Pitaveskii equation is recovered by setting $f(|u(x)|^2,x,t):= c/2\cdot |u(x)|^4 + V(t,x) \cdot |u(x)|^2$. \\

\section{Nonlinear Schr\"odinger equations of convolution type}

While the symplectic form $\omega$ is nondegenerate on $L^2(S^1,\IC)$, the Hamiltonian $H_t$ is only well-defined and smooth on its dense subspace $H^{1,2}(S^1,\IC)$. While this is apparent for the first summand as it involves the first derivative, observe that even the Hamiltonian $F_t$ modelling the nonlinearity is not defined on all of $\IH$ when the resulting Schr\"odinger equation is truely nonlinear. This in turn leads to true problems with the existence of the corresponding Hamiltonian flow $\phi_t=\phi^H_t$, describing the time-evolution of solutions of the nonlinear Schr\"odinger equation. \\

We start with the case of the \emph{free} nonlinear Schr\"odinger equation, that is, when the nonlinearity $f$ is equal to zero. Note that in this case the Hamiltonian $H_t$ simplifies to $$H^0(u)\,=\,\int_0^{2\pi} \frac{|u_x(x)|^2}{2}\;dx\,=\,\sum_{n=-\infty}^{+\infty} \frac{n^2}{2}|\hat{u}(n)|^2.$$ While the resulting Hamiltonian vector field $X^0(u)=i\Delta u$ is only defined on $H^{2,2}(S^1,\IC)$, we can prove the following result about the corresponding flow $\phi^0_t$.

\begin{proposition}\label{free-NLS} The flow of the free Schr\"odinger equation is given by $$\phi^0_t(u)\,=\,\exp(it\Delta)(u)\,=\,\sum_{k=0}^{\infty} \frac{(it)^k}{k!}\cdot \Delta^k(u).$$ For fixed time $t$ it preserves the $L^2$-norm and hence defines a linear symplectomorphism on the full symplectic Hilbert space $\IH=L^2(S^1,\IC)$, which restricts to a finite-dimensional linear symplectomorphism on every $\IC^{2k+1}:=\{u\in\IH:\,\hat{u}(n)=0\,\,\textrm{for all}\,\, |n|>k\}$. \end{proposition}

\begin{proof} In order to see this, observe that, after applying the Fourier transform, the symplectic vector field $X^0(\hat{u})(n)=in^2\cdot\hat{u}(n)$ has a linear flow given by $$\phi^0_t(\hat{u})(n)=\exp(itn^2)\cdot\hat{u}(n),\,\,n\in\IN.$$ Since in every frequency it multiplies the Fourier coefficient by a complex number of norm one, the claims follow. \end{proof}

On the other hand, if the Hamiltonian $F_t$ describing the nonlinearity is only densly defined, it is typically a very hard problem to establish the existence of a corresponding Hamiltonian flow $\phi^F_t$ on the full phase space $\IH$. The problem is that the flow on $\IH$ is no longer the unique solution of an ordinary differential equation given by the Hamiltonian vector field $X^F_t$. In order to circumvent problems arising from missing regularity in the nonlinear term, in this paper we hence work with a modification of the classical nonlinear Schr\"odinger equation, see \cite{K}.

\begin{definition} A  \emph{nonlinear Schr\"odinger equation of convolution type} is a Hamiltonian PDE with Hamiltonian $H_t=H^0+F_t$ on the symplectic Hilbert space $\IH=L^2(S^1,\IC)$, where the Hamiltonian defining the nonlinearity is now defined as $$F_t(u)\,:=\, \int_0^{2\pi} \frac{1}{2} f(|(u*\psi)(x)|^2,x,t)\,dx$$ with $(u*\psi)(x)=\<u,\psi(x-\cdot)\>_{\IC}$, where $\psi\in C^{\infty}(S^1,\IR)$ is some fixed smoothing kernel, and $f$ still denotes a smooth, real-valued function on $\IR^+\times S^1\times\IR$. \end{definition}

Instead of considering density functions of the form $f(|u(x)|^2,x,t)$, from now on we hence consider density functions of the form $f(|(u*\psi)(x)|^2,x,t)$. Nonlinear Schr\"odinger equations of convolution type describe multi-particle systems with nonlocal interaction. The comparison with the above example and a short computation show that the resulting nonlinear Schr\"odinger equations are given by $$i\del_t u \,=\, - \Delta u \;+\; \del_1  f(|(u*\psi)(x)|^2,x,t)(u*\psi)*\psi,$$ where $\del_1 f$ means derivative with respect to the first coordinate. \\

We collect important observations about these class of equations in the following  

\begin{proposition} For every nonlinear Schr\"odinger equation of convolution type the resulting flow is given by the composition $\phi_t=\phi^G_t\circ\phi^0_t$ of the flow $\phi^0_t$ of the free Schr\"odinger equation with the Hamiltonian flow of $G_t=F_t\circ\phi^0_{-t}$. The time-dependent Hamiltonian functions $F:\IR\times\IH\to\IR$ as well as $G:\IR\times\IH\to\IR$ are smooth; in particular, for fixed time $t$ the map $\phi_t:\IH\to\IH$ is a smooth symplectomorphism defined on the full Hilbert space $\IH=L^2(S^1,\IC)$. Furthermore, everything descends to the projective Hilbert space $\IP(\IH)$ equipped with the Fubini-Study form. \end{proposition}

\begin{proof} Since the convolution of a function $u\in L^2(S^1,\IC)$ with the smooth function $\psi$ ensures that the resulting function is smooth, i.e., $u*\psi\in C^{\infty}(S^1,\IC)$, it immediately follows that the Hamiltonian $F$ is well-defined and smooth on all of $\IR\times\IH$. Furthermore, since $\phi^0_{-t}(u)*\psi=\phi^0_{-t}(u*\psi)=u*\phi^0_{-t}(\psi)$ and $\phi^0_t$ is infinitely often differentiable with respect to $t$ on $C^{\infty}(S^1,\IC)\subset\IH$, it follows that the same continues to hold for $G$. \\
 
On the other hand, while it immediately follows from the fact that $\phi^0_t$ is a unitary linear map that $\phi^0_t$ descends to the projectivization $\IP(\IH)$ of $\IH$, in order to prove the same for the Hamiltonian flows of $F$ (and hence of $G$), we first introduce the new Hamiltonian function $K:\IH\to\IR$ given by (half the square of) the Hilbert space norm, $K(u)=|u|^2/2$. With this it remains to prove that the Hamiltonian flow of $F$ preserves the Hamiltonian $K$ and vice versa. Since $$X^K(F)=-X^F(K)=\omega(X^F,X^K)=\<X^F,\nabla K\>_{\IR}$$ wiht $\nabla K(u)=u$, it suffices to show that, at every point $u\in\IH$, the symplectic gradient $X^F(u)=X^F_t(u)=i\del_1 f(|(u*\psi)(x)|^2,x,t)(u*\psi)*\psi\in\IH$ is perpendicular to $u$ with respect to the real inner product on $\IH=L^2(S^1,\IR^2)$. First observe that the statement is immediately clear if $u$ is truely real (that is, $u(x)\in\IR\subset\IC$ for all $x\in S^1$) or truely imaginary, as in this case $X^F(u)$ is truely imaginary, or truely real, respectively. For the general case write $u=u_R+i u_I$ and $X^F(u)=X^F_R(u)+i X^F_I(u)$ with real-valued functions $u_R,u_I,X^F_R,X^F_I$. In this case we can use the compatibility of the real $L^2$-inner product with the product and convolution of functions to show that $\<u_R,\del_1 f(|(u*\psi)(x)|^2,x,t)(u_I*\psi)*\psi\>$ is equal to $\<u_I,\del_1 f(|(u*\psi)(x)|^2,x,t)(u_R*\psi)*\psi\>$, which in turn proves that $$\<u,X^F(u)\>\,=\,\<u_I,X^F_R(u)\>\;-\;\<u_R,X^F_I(u)\>\,=\,0,$$ finishing the proof of the proposition. \end{proof}

In what follows we view the projective Hilbert space as quotient of the unit sphere $\IS(\IH)$ in $\IH$ by the action of $U(1)=S^1$, $\IP(\IH)=\IS(\IH)/S^1$. Note that studying the Schr\"odinger equation on $\IP(\IH)$ in place of $\IH$ is also natural from the view point of quantum physics. \\

Before we can state the main theorem, we however first need to introduce a small technical assumption, which will play the role of a nondegeneracy condition for the Hamiltonian in infinite dimensions. \\

Following the proof of proposition \ref{free-NLS}, the complex eigenvalues of the (linear) time-one flow map $\phi^0_1=\exp(i\Delta)$ are given by the sequence $\exp(in^2)\in\IC$, $n\in\IN$. After restricting to a finite-dimensional subspace $\IC^{2k+1}\subset\IH$ and passing to the projectivization, it follows that all fixed points of $\phi^0_1$ are nondegenerate in the sense that one is not an eigenvalue of the linearized return map. On the other hand, after passing to infinite-dimensional case, the latter is no longer true, as a subsequence of eigenvalues may converge to $1$, a consequence of the fact that the ratio of the spatial and of the time period is irrational. In order to avoid resulting problems with the lack of nondegeneracy of the time-one flow map of the free Schr\"odinger equation, we restrict ourselves to smoothing kernels $\psi$ which are special in the sense of the following 

\begin{definition}\label{admissible} A smoothing kernel $\psi\in C^{\infty}(S^1,\IR)$ is said to be \emph{admissible} if the following holds: Denoting by $\hat{\psi}:\IZ\to\IC$ the Fourier transform of $\psi:S^1\to\IR$, we require that there exists some $\delta>0$ such that $\hat{\psi}(n)=0$ whenever $|\exp(in^2)-1|<\delta$ for all frequencies $n\in\IZ$. \end{definition}

We emphasize that the threshold $\delta>0$ can be chosen arbitrarily small and that every finite-dimensional nonlinearity, that is, where $\textrm{supp}(\hat{\psi})\subset\{-\ell,\ldots,+\ell\}$ for some $\ell\in\IN$, is clearly admissible. On the other hand, for every $\delta<2$ it follows by the same arguments that the spectrum of allowed frequencies $M_{\delta}:=\{n\in\IZ:\;|\exp(in^2)-1|\geq\delta\}$ is unbounded in both directions. Everything relies on the fact that the quotient of the spatial period ($=2\pi$) and the time period ($=1$) is irrational; in particular, we claim that everything can be generalized to other spatial and time periodicities, as long as the latter irrationality is preserved.    

\section{Statement of the main theorem}

From now on let us assume that the nonlinear term in the Schr\"odinger equation is one-periodic in time, that is, $f(|(u*\psi)(x)|^2,x,t+1)=f(|(u*\psi)(x)|^2,x,t)$. Then it follows that every nonlinear Schr\"odinger equation defines a flow $\phi_t=\phi^H_t$ on the projective Hilbert space $\IP(\IH)$ where the underlying Hamiltonian is one-periodic in time, $H_{t+1}=H^0+F_{t+1}=H^0+F_t=H_t$. \\

In the case of time-one-periodic smooth Hamiltonians on finite-dimensional projective spaces $\CP^n=\IP(\IC^{n+1})$ we have the following famous   

\begin{theorem} (\cite{Fo},\cite{Sch}) The time-one map of a Hamiltonian flow on $\CP^n$ always has at least $n+1$ fixed points, that is, the degenerate version of the famous Arnold conjecture holds. \end{theorem}

Viewing the Hamiltonian flow $\phi_t$ on $\IP(\IH)$ defined by the nonlinear Schr\"odinger equation of convolution type as an infinite-dimensional generalization, it is natural to ask whether an analogue of the Arnold conjecture also holds in this infinite dimensional context, establishing the existence of infinitely many fixed points of the time-one map. \\

But first, in order to show that the generalization to infinite dimensions is not trivial and we can only expect it to hold after imposing restrictions, we first give the following counterexample.

\begin{proposition} There exists a smooth Hamiltonian function on $\IP(\IH)$ whose time-one map has no fixed points at all.  \end{proposition}

\begin{proof} The function $L$ defined by $$L:\IH\to\IR,\,\,L(u) \,:=\, \int_0^{2\pi} V(x) \frac{|u(x)|^2}{2}\;dx$$ decends to a function on the symplectic quotient $\IP(\IH)=\IS(\IH)/S^1$, since its flow map is given by $(\phi^L_t(u))(x) = \exp(it V(x))\cdot u(x)$ and hence preserves the $L^2$-norm. In the same way it can be seen that $u\in\IS(\IH)$ is a fixed point of $\phi^L_1$ on $\IP(\IH)$ if and only if there exists some $a\in\IR$ such that for all $x\in S^1$ we have $\exp(iV(x))\cdot u(x)=\exp(ia)\cdot u(x)$ and hence either $(V(x)-a)/2\pi\in\IZ$ or $u(x)=0$. For a generic choice of the function $V:S^1\to\IR$ it follows that $u(x)=0$ almost everywhere and hence $u=0\not\in\IS(\IH)$, resulting in the fact that its time-one map on $\IP(\IH)$ has no fixed points at all. \end{proof}

Like for the linear Schr\"odinger equation, in our proof the appearance of the Laplace term in the nonlinear Schr\"odinger equations turns out to be essential to find infinitely many fixed points. Before we turn to the general case, we first have a look at the free Schr\"odinger equation where $F_t=0$. The proof of the following proposition is an easy exercise.

\begin{proposition} After passing to the projectivization, the time-one flow map $\phi^0=\phi^0_1$ of the free Schr\"odinger equation on $\IP(\IH)$ has infinitely many different fixed points $u_n^0$ given by the complex oscillations, $$u^0_n: S^1\to \IC,\,u_n^0(x)=\frac{1}{\sqrt{2\pi}}\exp(inx)\,\,\textrm{for every}\,\,n\in\IN.$$ \end{proposition}

From now on let $F_t$, $G_t=F_t\circ\phi^0_{-t}$ and $\phi^0_t$, $\phi_t=\phi^G_t\circ\phi^0_t$ denote the corresponding functions and flows on the projective Hilbert space $\IP(\IH)=\IS(\IH)/S^1$. Furthermore recall that the Hofer norm of the time-periodic Hamiltonian $F_t$ is defined as $$|||F|||\,:=\,\int_0^1 (\max F_t - \min F_t)\;dt$$  By generalizing Floer theory to the case of infinite-dimensional symplectic manifolds, in this paper we prove the following infinite-dimensional version of Floer's proof of the Arnold conjecture. 

\begin{theorem} Assume that, after descending to the projectivization, the Hofer norm $|||F|||$ of the Hamiltonian defining the nonlinearity is smaller than $\pi/4$ and that the underlying smoothing kernel is admissible. Then for every fixed point $u^0_n$ of the time-one map $\phi^0_1$ of the free Schr\"odinger equation there exists a fixed point $u^1_n$ of the time-one map $\phi_1$ of the given nonlinear Schr\"odinger equation of convolution type, and a Floer strip $\tilde{u}_n:\IR\times [0,1]\to\IP(L^2(S^1,\IC))$ which connects $u^0_n$ and $u^1_n$. Furthermore one obtains infinitely many different fixed points of $\phi_1:\IP(L^2(S^1,\IC))\to\IP(L^2(S^1,\IC))$ this way as $u^1_m\neq u^1_n$ for all $m>n\geq 4$. \end{theorem}

We first explain the statement about the existence of a Floer strip connecting the fixed point $u^0_n$ of the free Schr\"odinger equation with a fixed point $u^1_n$ of the given Schr\"odinger equation of convolution type, which we view as a path $u^1_n:[0,1]\to\IP(\IH)$ with $u^1_n(1)=\phi^0(u^1_n(0))$ and $\del_t u^1_n=X^G_t(u^1_n)$. With this we mean a smooth map $\tilde{u}=\tilde{u}_n:\IR\times [0,1]\to\IP(\IH)$ with $\tilde{u}(\cdot,1)=\phi^0(\tilde{u}(\cdot,0))$ satisfying the Floer equation $$0\,=\,\CR\tilde{u}\;-\;\varphi(s)\cdot\nabla G_t(\tilde{u}),$$ where $\CR=\del_s+i\del_t$ denotes the standard Cauchy-Riemann operator and $\varphi$ is a smooth cut-off function with $\varphi(s)=0$ for $s\leq -1$ and $\varphi(s)=1$ for $s\geq 0$. It connects $u^0_n$ and $u^1_n$ in the sense that there exist two sequences $(s_{\alpha}^{\pm})$ of real numbers, $s_{\alpha}^{\pm}\to\pm\infty$ with $\tilde{u}_n(s_{\alpha}^-,\cdot)\to u^0_n$,  $\tilde{u}_n(s_{\alpha}^+,\cdot)\to u^1_n$ as $\alpha\to\infty$; the latter weaker asymptotic condition is a consequence of the fact that we do \emph{not} want to assume that the nonlinearity is generic in the sense that all orbits are isolated. \\   

We will show below that the Hofer norm is indeed always finite for nonlinearities of convolution type, so that the condition can always be fulfilled after rescaling the function $f$ or, equivalently, rescaling the time or space variables. Furthermore, we do not claim that the bound on the Hofer norm is sharp in any sense; indeed we only want to stress the fact that this is  \emph{not} a perturbative result in the spirit of KAM theory.  

\begin{remark} Before we turn to the proof of the main theorem, there are a few bibliographical remarks in order:
\begin{itemize} 
\item[i)] The requirement for the smoothing kernel to be admissible is related to the famous small divisor problem that plays an important role in KAM theory, see \cite{EK} for the case of the nonlinear Schr\"odinger equation.
\item[ii)] While KAM theory for Hamiltonian PDE's has already been extensively studied, the only global (i.e., non-perturbative) results that are known to the author are generalizations of the seminal work of P. Rabinowitz in \cite{Ra} on the existence of time-periodic solutions of the nonlinear wave equation, see \cite{B} for an overview. In order to avoid the problem with small divisors, Rabinowitz only considers the case where the time period is a natural multiple of the space period.
\item[iii)] While a general nonsqueezing theorem for symplectomorphisms in infinite dimensions is not proven, a significant amout of work has gone into proving analogues of Gromov's theorem for special Hamiltonian PDE (e.g. \cite{K} for the case of nonlinear wave equations and nonlinear Schr\"odinger equations of convolution type) and under natural assumptions (e.g. \cite{AM},\cite{Fa} and the references therein).
\end{itemize}\end{remark}

Forgetting for the moment that we are working in the setting of infinite-dimensional symplectic manifolds, following Gromov's existence proof of symplectic fixed points in \cite{Gr} the idea would be to study moduli spaces of Floer strips $(\tilde{u},T)$, where $T\in\IR^+_0$ is some non-negative real number and $\tilde{u}=\tilde{u}_{n,T}$ now denotes a smooth map $\tilde{u}:\IR\times [0,1]\to\IP(\IH)$ with $\tilde{u}(\cdot,1)=\phi^0(\tilde{u}(\cdot,0))$, which now satisfies the asymptotic condition $\tilde{u}(s,t)\to u^0_n$ as $s\to\pm\infty$ as well as the $T$-dependent perturbed Cauchy-Riemann equation $$\CR^T_G\tilde{u}\,=\,\CR\tilde{u} \;-\; \varphi_T(s)\cdot\nabla G_t(\tilde{u})\,=\, 0,$$ where $\varphi_T:\IR\to [0,1]$ now denotes a family of smooth cut-off functions with compact support with $\varphi_T(s)=1$ for $s\in [0,2T]$. Assuming that, as in the case of finite-dimensional projective spaces, one could compactify the above moduli space by just adding broken holomorphic strips corresponding to the case that $T$ converges to $+\infty$, see \cite{DS} and the references therein, one would be able to show that for every $n\in\IN$ there exists a Floer strip connecting the fixed point $u^0_n$ of the free Schr\"odinger equation with a fixed point $u^1_n$ of the given Schr\"odinger equation of convolution type. In order to see that there are indeed infinitely many different fixed points $u^1_n$, we will use bounds for their symplectic action. \\

On the other hand, it is quite apparent that the underlying theory of pseudo-holomorphic curves does not instantly carry over from finite to infinite dimensions. In particular, the non-compactness of the target manifold leads to the fact that Gromov's compactness theorem does not naturally generalize from finite-dimensional projective spaces to $\IP(\IH)$. Apart from the fact that the nonsqueezing problem in infinite dimensions has already received some attention over the recent years, the field of infinite-dimensional symplectic geometry is indeed much less explored than its finite-dimensional counterpart. In this paper we will make use of the established fact that non-standard model theory, introduced by A. Robinson in his seminal book \cite{R}, provides a very efficient way to translate results from the finite to the infinite context, see e.g. the paper \cite{Os} on infinite-dimensional Brownian motion. We hope that this paper as well as our first paper \cite{Fa} on Gromov's nonsqueezing result will serve as a starting point of a general program to generalize results from finite-dimensional symplectic geometry to infinite dimensions using this tool. While a standard approach would require to prove appropriate  infinite-dimensional generalizations of every technical result, in particular, the bubbling-off phenomenon and elliptic bootstrapping, our non-standard proofs only build on the well-established finite-dimensional results. Note that this fits nicely with the well-established general promise of non-standard methods to drastically help to find proofs. \emph{We would like to emphasize that this paper is written in such a way that it does not require any previous knowledge of non-standard model theory.}\\

Indeed it is well-known, see e.g. \cite{L1}, \cite{L2} and \cite{Ke}, that there exist so-called \emph{non-standard} models of mathematics in which there exists an extension of the notion of finiteness: There exist new so-called unlimited *-real (and *-natural) numbers which are greater than all \emph{standard} real (and natural) numbers; in an analogous way there exist infinitesimal numbers, whose moduli are smaller than any positive standard real number. These *-real numbers can be introduced, using the axiom of choice, as equivalence classes of sequences of real numbers, where the standard numbers are included as constant sequences, while sequences converging to $\pm\infty$ or $0$ are examples of unlimited and infinitesimal numbers, respectively. In this paper we will use the resulting surprising fact that there exists a *-finite-dimensional symplectic space $\IF$, i.e., a finite-dimensional symplectic space in the sense of the non-standard model, which contains the infinite-dimensional Hilbert space $\IH$ as a subspace.\footnote{To be more precise, $\IH$ will be contained in $\IF$ 'up to an infinitesimal error'} Furthermore, after passing to the projectivizations $\IP(\IH)\subset\IP(\IF)$, the smooth infinite-dimensional Hamiltonian flow $\phi_t=\phi^{\IH}_t$ can be represented by a *-finite-dimensional Hamiltonian flow $\phi^{\IF}_t$. For the latter we are going to use that nonlinear Schr\"odinger equations of convolution type can be uniformly approximated by finite-dimensional Hamiltonian flows. While the latter would also be the starting point for a standard proof, one still would need to generalize the full Gromov-Floer compactness result to the case of infinite-dimensional target manifolds. In sharp contrast to this, our non-standard proof entirely relies on well-established finite-dimensional results as we are still working in a finite-dimensional setup (in the sense of the non-standard model). \\

While the existence of ideal elements such as $\IF$ is established using the so-called \emph{saturation principle}, the so-called \emph{transfer principle} ensures that every statement that holds in finite dimensions and can be formulated in first-order logic has an analogue in the *-finite-dimensional setting. In particular, we will employ that, for every fixed point $u^0_n\in\IP(\IH)\subset\IP(\IF)$ of the free nonlinear Schr\"odinger equation, the transfer principle will provide us with the existence of a Floer strip $\tilde{u}_n$ in $\IP(\IF)$ in the non-standard sense. In order to show that we actually obtain a Floer strip in $\IP(\IH)$ in the standard sense as stated in the main theorem, we first use the non-standard analogues of bubbling-off analysis and elliptic regularity in order to prove that the non-standard derivatives of $\tilde{u}_n$ have limited norm. The latter result then does not only allow us to employ a non-standard minimal surface argument in order to prove that the image of $\tilde{u}_n$ has to lie in an infinitesimal neighborhood of $\IP(\IH)$, but also can be used to show that the resulting standard Floer strip in $\IP(\IH)$ is indeed smooth in the standard sense. Since for our proof we are never leaving the finite-dimensional setup, we emphasize that our non-standard proof only builds on the corresponding well-established results from finite dimensions, combined with the transfer principle for non-standard models. \\

Summarizing we hence use that the existence of the Floer strips is guaranteed in the abstract set $\IP(\IF)$ and we only show afterwards, employing transferred versions of further standard finite-dimensional results, that all appearing norms are indeed finite, in particular, the Floer strips indeed sit in the correct set $\IP(\IH)$. Informally speaking, viewing $\IP(\IF)$ as an enlargement of $\IP(\IH)$ obtained by considering infinite sums without requiring the Hilbert norm to be finite, the existence of the Floer strip in $\IP(\IF)$ follows word by word from the proof in finite dimensions by allowing every number to be possibly infinite or infinitesimal, while finiteness is only shown a posteriori. The main advantage of our non-standard approach over a standard approach is hence that we do not need to keep track of finiteness of norms in each technical lemma which would be needed in order to generalize it from finite-dimensional Euclidean space to an infinite-dimensional Hilbert space. \\ 

This paper is organized as follows: While in section $4$ we show how finite-dimensional symplectic Floer theory can be used to prove our main theorem in the special case of so-called finite-dimensional nonlinearities, for the case of general infinite-dimensional nonlinearities we provide a return ticket back to the finite-dimensional case using non-standard models in section $5$, where we carefully introduce all relevant concepts and refer to the appendix for further details. Finally, in section $6$ we prove our main theorem in the general case by combining the Floer theoretic results from section $4$ with the non-standard model theoretic results from section $5$. More precisely, in the proof we give for every desired non-standard result the corresponding standard result from finite-dimensional Floer theory and mention the used non-standard tool listed in the appendix. 

\section{Floer strips in complex projective spaces}

Before we turn to the general case, we first restrict ourselves to the case of finite-dimensional nonlinearities. Based on Floer's proof of the Arnold conjecture in finite dimensions we show

\begin{proposition}\label{finite-dim} Assume the nonlinearity is finite-dimensional in the sense that the support of the Fourier transform $\hat{\psi}:\IZ\to\IC$ of the smoothing kernel $\psi$ is finite, that is, $\textrm{supp}(\hat{\psi})\subset\{-k,\ldots,+k\}$ for some natural number $k$. Then the statement of the main theorem holds. \end{proposition}

Identifying the symplectic Hilbert space $\IH$ with $\ell^2(\IC)$ using the Fourier transform, it follows that $G$ just depends on its value after applying the projection $\pi_k:\IH\to\IC^{2k+1}$ onto the finite-dimensional symplectic subspace $\IC^{2k+1}=\{\hat{u}\in\IH: \hat{u}(n)=0\,\,\textrm{for all}\,\,|n|>k\}$. In other words we have $G=G^k:=G\circ\pi_k$, so that at every point  the gradients $\nabla G_t$ and $X^G_t$ are vectors in $\IC^{2k+1}\subset\IH$. Note that, together with proposition \ref{free-NLS},  this implies that the flow $\phi_t$ on $\IP(\IH)$ restricts to a symplectic flow $\phi^k_t$ on $\CP^{2k}$. \\

With this the proof essentially relies on the following existence result of Floer strips in finite-dimensional complex projective spaces. From now on let us assume that the Hofer norm $|||F|||$ of $F_t$ on $\IP(\IH)$ is strictly smaller than $\pi/4$. Furthermore, let $\varphi_T:\IR\to [0,1]$, $T\in\IR^+\cup\{0\}$ denote a smooth family of smooth cut-off functions with $\varphi_0=0$ and $\varphi_T(s)=0$ for $s\leq -1$, $s\geq 2T+1$ and $\varphi_T(s)=1$ for $0\leq s\leq 2T$  for $T\geq 1$. Furthermore the natural Riemannian norm on $\CP^{2k}$ is denoted by $|\cdot|$.

\begin{proposition}\label{strips} Let $k\in\IN$ and $T\in\IR^+\cup\{0\}$. Then for every $n\leq k$ there exists a smooth map $\tilde{u}=\tilde{u}_n=\tilde{u}^k_{n,T}:\IR\times [0,1]\to\CP^{2k}$, called \emph{Floer strip}, satisfying the periodicity condition $\tilde{u}(\cdot,1)=\phi^0_1(\tilde{u}(\cdot,0))$, the asymptotic condition $\tilde{u}_n(s,\cdot)\to u^0_n$ as $s\to\pm\infty$, and the perturbed Cauchy-Riemann equation $$0=\CR^T_G\tilde{u}=\CR\tilde{u} - \varphi_T(s)\cdot \nabla G^k_t(\tilde{u}).$$ Furthermore, for the resulting families of maps $\tilde{u}=\tilde{u}_n$ we have that the energy $E(\tilde{u}_n)$ defined by $$E(\tilde{u}_n):=\int_{-\infty}^{+\infty}\int_0^1 \frac{1}{2}\big(|\del_s\tilde{u}_n|^2 + |\del_t\tilde{u}_n - \varphi_T(s)\cdot X^{G,k}(\tilde{u}_n)|^2\big)\;dt\;ds$$ is bounded by $2 |||G^k|||<\pi/2$. \end{proposition} 

\begin{proof} For the proof one observes that, for every $k\in\IN$, $\CP^{2k}$ is a closed symplectic manifold where the energy of a holomorphic sphere is bounded from below by $\pi$. Since, for every $n\leq k$, $u^0_n\in\IP(\IH)$ given by $u^0_n(x)=\exp(inx)$ is a fixed point of $\phi^0_1$ in $\CP^{2k}$, it follows that the existence of the map $\tilde{u}=\tilde{u}^k_{n,T}$ for every $T\in\IR^+\cup\{0\}$ can be deduced using the properties of the moduli space of Floer strips $(\tilde{u},T)$ established in \cite{DS} and the references therein; an alternative proof is provided by \cite{BEHWZ} by translating everything into a statement about holomorphic curves in almost complex manifolds with cylindrical ends.\\

First we emphasize that an easy computation shows that the standard complex structure satisfies the periodicity condition $(\phi^0_1)_* i= i$ required in \cite{DS}. Assuming for the moment that transversality for the Cauchy-Riemann operator $\CR_G$ given by $\CR_G(\tilde{u},T)=\CR^T_G\tilde{u}$ holds, it follows that the moduli space of tuples $(\tilde{u},T)$ is a one-dimensional manifold. Since for $T=0$ the constant strip $\tilde{u}^k_{n,0}(s,t)=u^0_n$ staying over the fixed point $u^0_n\in \CP^{2k}$ of $\phi$ is the unique solution, the existence of a Floer strip $\tilde{u}^k_{n,T}$ for all $T\in\IR^+$ follows from the Gromov-Floer compactness result. Here we emphasize that bubbling-off of holomorphic spheres is excluded due to fact that the energy $E(\tilde{u})$ is bounded from above by twice the Hofer norm of the Hamiltonian $G^k$, and the Hofer norm of $G^k$ is smaller than 1/2 the minimal energy of a holomorphic sphere in $\CP^{2k}$; for an analogous statement see proposition 9.1.4 in \cite{MDSa}. Finally, since the Cauchy-Riemann operator $\CR_G$ cannot be expected to be transversal, one first has to approximate $i$ by a family of $t$-dependent almost complex structures $J^{\nu}_t$, $t\in [0,1]$ with $J^{\nu}_1=(\phi^0_1)_* J^{\nu}_0$ in the sense that $J^{\nu}_t\to J^0_t=i$ as $\nu\to 0$. Then the Gromov-Floer compactness result in \cite{DS} can be used again to deduce the existence of Floer strips for $\nu=0$ from the existence of Floer strips for $\nu\neq 0$ for all $T\in\IR^+$. In particular, we emphasize that one does not need more elaborate technology like Kuranishi structures or polyfolds to establish the desired properties of the moduli space. \\

Concerning the additional statement, observe that the bound on $E(\tilde{u})$ has already been used to exclude bubbling-off for compactness and it can be found in \cite{MDSa} (for the case of $\phi$ being the identity). It relies on the fact that, since $\CR^T_G\tilde{u}=0$, the energy of $u$ is given by $$E(\tilde{u})=\int_{-\infty}^{+\infty}\int_0^1 \varphi_T(s) \<\nabla G^k_t(\tilde{u}),\del_s\tilde{u}\>\;dt\;ds. $$ \end{proof}

Note that, in the case of the free Schr\"odinger equation, the fixed points $u^0_n$ of the time-one map $\phi^0$ are distinguished by their symplectic action, defined in \cite{DS}. For this choose $u^0_0=1$ as reference fixed point and choose for each $u^0_n$ a holomorphic strip $\tilde{u}_{0n}:\IR\times [0,1]\to\IP(\IH)$ with $\tilde{u}_{0n}(\cdot,1)=\phi^0(\tilde{u}_{0n}(\cdot,0))$, $\tilde{u}_{0n}(s,\cdot)\to u_n^0$ as $s\to +\infty$ and $\tilde{u}_{0n}(s,\cdot)\to u_0^0$ as $s\to -\infty$ given by a gradient flow line $u_{0n}:\IR\to\IP(\IH)$ of $H^0$ by $\tilde{u}_{0n}(s,t)=\phi^0_t(u_{0n}(s))$. Then it can be shown that the symplectic action $\AA(u^0_n)$ defined in \cite{DS} is given by the Hamiltonian $H^0$ itself, 
\begin{eqnarray*} \AA(u^0_n) &:=& \int_{-\infty}^{+\infty} \int_0^1 \omega(\del_s\tilde{u}_{0n},\del_t\tilde{u}_{0n})\;dt\;ds\\ &=& \int_{-\infty}^{+\infty} \omega(\del_s u_{0n},i\nabla H(u_{0n}))\;ds\\ &=& \int_{-\infty}^{+\infty} \del_s(H\circ u_{0n})\;ds\\ &=& H^0(u^0_n) - H^0(u^0_0)\end{eqnarray*} with $H^0(u^0_n)=n^2/2$, in particular, it grows quadratically with $n\in\IN$. Note that here we use the translation between Hamiltonian Floer theory and Floer theory for general symplectomorphisms. Furthermore we are building on the fact that $H^0$ restricts to a smooth function $$H^{0,k}(u)\,=\,\sum_{n=-k}^{+k}\frac{n^2}{2}|\hat{u}(n)|^2$$ on a finite-dimensional projective space $\CP^{2k}\subset\IP(\ell^2(\IC))$ and every gradient flow line connecting $u^0_0$ and $u^0_n$ indeed stays inside this $\CP^{2k}$ as long as $n\leq k$. \\

In order to prove that we obtain sufficiently many different fixed points of the time-one map $\phi_1$ of the given nonlinear Schr\"odinger equation of convolution type in the end, we want to make use of the fact that, just as in the case of the free Schr\"odinger equation, they also can be distinguished by their symplectic action. This is the content of the following 

\begin{proposition}\label{action} For every $s\in [0,2T]$ the symplectic action $\mathcal{A}(\tilde{u}_n(s)):=\mathcal{A}^n(\tilde{u}_n(s))$ of the path $\tilde{u}_n(s)=\tilde{u}_n(s,\cdot): [0,1]\to\CP^{2k}$, defined as $$\int \tilde{u}_{0n}^*\omega \,+\, \int \big(\tilde{u}_n|_{(-\infty,s)\times [0,1]}\big)^*\omega \,+\, \int_0^1 G^k_t(\tilde{u}(s,t))\;dt,$$  is well-defined in the sense that it is independent of $n\in\IN$. Furthermore, for $m>n\geq 4$ it holds that $|\mathcal{A}(\tilde{u}_m(s))-\mathcal{A}(\tilde{u}_n(s'))|>1$. \end{proposition}

\begin{proof} For the definition of the symplectic action following \cite{DS} we use that, after concatenation with the strip $\tilde{u}_{0n}:\IR\times [0,1]\to\CP^{2k}$ used in the definition of the symplectic action $\AA(u^0_n)$, the Floer strip $\tilde{u}=\tilde{u}_n$ can be used to connect $\tilde{u}_n(s)$ with $u^0_0$. For the well-definedness let us assume that $\tilde{u}_m(s)=\tilde{u}_n(s')$. Then one has to show that the concatenated strips $\tilde{u}_{0m}\# \tilde{u}_m|_{(-\infty,s]}$ and $\tilde{u}_{0n}\# \tilde{u}_n|_{(-\infty,s']}$ are homotopic, that is, the homotopical difference between $\tilde{u}_m|_{(-\infty,s]}$ and the concatention of a holomorphic strip $\tilde{u}_{nm}:\IR\times [0,1]\to\CP^{2k}$ connecting $u^0_n$ with $u^0_m$ given by a gradient flow line of $H^0$ with $\tilde{u}_n|_{(-\infty,s']}$ represents the zero class in $\pi_2(\CP^{2k})$. But since every gradient flowline is contractible, it suffices to observe that the energy bound in proposition \ref{strips} implies that $$-\pi\,<\,\int \big(\tilde{u}_m|_{(-\infty,s]}\big)^*\omega \,-\, \int \big(\tilde{u}_n|_{(-\infty,s']}\big)^*\omega \,<\, +\pi,$$ so that the homotopical difference still has to be represented by the zero class in $\pi_2(\CP^{2k})$. For the last statement note that the same proof as used to establish the energy bound proves the inequality $\big|\mathcal{A}^n(\tilde{u}_n(s))-n^2/2\big|\,\leq\, 2\cdot |||G^k|||\,<\,\pi/2.$ Together with the fact that $\AA(u^0_m)-\AA(u^0_n)=m^2/2 - n^2/2>\pi+1$ for all $m>n\geq 4$, we find that $|\mathcal{A}(\tilde{u}_m(s)-\mathcal{A}(\tilde{u}_n(s'))|>1$, in particular, $\tilde{u}_m(s)\neq \tilde{u}_n(s')$ for different $m>n\geq 4$.  \end{proof}

With this we can show that the main theorem holds for true in the case of finite-dimensional nonlinearities. 

\begin{proof}\emph{(of proposition \ref{finite-dim})} For this one has to use that the Gromov-Floer compactness theorem furthermore proves that, as $T$ tends to $+\infty$, the sequence $\tilde{u}_T=\tilde{u}^k_{n,T}$ of Floer strips converges to a $d$-times broken Floer strip $\tilde{u}_{\infty}=(\tilde{u}^1_{\infty},\ldots,\tilde{u}^d_{\infty})$ for some $d\in\IN$, where $\tilde{u}:=\tilde{u}^1_{\infty}:\IR\times [0,1]\to\CP^{2k}$ is a smooth map with $\tilde{u}(\cdot,1)=\phi^0_1(\tilde{u}(\cdot,0))$ satisfying the Floer equation $0=\CR\tilde{u}-\varphi(s)\cdot\nabla G_t(\tilde{u}),$ where $\varphi$ is now a smooth cut-off function with $\varphi(s)=0$ for $s\leq -1$ and $\varphi(s)=1$ for $s\geq 0$. It connects the fixed point $u^0_n$ of the free Schr\"odinger equation with a fixed point $u^1_n$ of $\phi_1$ in the sense that there exist sequences $(s_{\alpha}^{\pm})$ of real numbers, $s_{\alpha}^{\pm}\to\pm\infty$ with $\tilde{u}(s_{\alpha}^-,\cdot)\to u^0_n$ and $\tilde{u}(s_{\alpha}^+,\cdot)\to u^1_n$ as $\alpha\to\infty$. In order to see the latter, note that by the bound for the energy $E(\tilde{u})$ from proposition \ref{strips}, it follows that for every $\alpha\in\IN$ there exists $\alpha\leq |s_{\alpha}|\leq 2\alpha$ such that $$\int_0^1 |\del_t\tilde{u}(s_{\alpha}^{\pm},t) - \varphi(s_{\alpha}^{\pm})X^{G,k}_t(\tilde{u}(s_{\alpha}^{\pm},t))|^2\;dt\,<\,\frac{\pi}{2\alpha}.$$ By compactness of $\CP^{2k}$ we know, possibly after passing to a subsequence, that the sequence $\tilde{u}(s_{\alpha}^{\pm},0)$ converges to a fixed point of $\phi^0_1$ or $\phi_1$, respectively. Recall that this weaker asymptotic condition is a consequence of the fact that we do not want to assume that the nonlinearity is generic in the sense that all orbits are isolated. In order to see that $u^1_m\neq u^1_n$ for all $m>n\geq 4$, observe that the inequality $|\mathcal{A}(\tilde{u}_{m,T}(s))-\mathcal{A}(\tilde{u}_{n,T}(s'))|>1$ for all $T>0$ and $s,s'\in[0,2T]$ implies that $|\mathcal{A}(u^1_m)-\mathcal{A}(u^1_n)|>1,$ where the action of $u^1_n$ is defined as $$\mathcal{A}(u^1_n)\,=\,\int \tilde{u}_{0n}^*\omega \,+\, \int \tilde{u}_n^*\omega \,+\, \int_0^1 G^k_t(\tilde{u}(s,t))\;dt,$$ and $\tilde{u}_n$ now denotes the Floer strip connecting $u^0_n$ and $u^1_n$; furthermore this definition of action is independent of $n\in\IN$ by the same arguments as used in the proof of proposition \ref{action}. \\

On the other hand, for $m>n$ it follows that the fixed points $u^0_n$ with $n>k$ of the free Schr\"odinger equation are also fixed points of $\phi_1$, where the corresponding Floer strip $\tilde{u}=\tilde{u}_n$ is just the constant strip $\tilde{u}_{n,0}(s,t)=u^0_n=u^1_n$, $(s,t)\in\IR\times [0,1]$. \end{proof}

In the case of finite-dimensional nonlinearities we see that the main theorem can be proven by studying Floer curves in finite-dimensional complex projective spaces. In preparation for the general statement, we first show that everything is independent of the chosen ambient finite-dimensional projective space. 

\begin{proposition}\label{Liouville} Assume that $\supp(\hat{\psi})\subset\{-\ell,\ldots,+\ell\}$. For $k\geq\ell\geq n$ and $T>0$  let $\tilde{u}=\tilde{u}^k_{n,T}:\IR\times [0,1]\to\CP^{2k}$ be a Floer strip as in proposition \ref{strips}. Then we have $\tilde{u}(s,t)\in\CP^{2\ell}\subset\CP^{2k}$ for all $(s,t)\in\IR\times [0,1]$. \end{proposition} 

\begin{proof} Let $\tilde{u}^{\ell}=\pi_{\ell}\circ \tilde{u}:\IR\times [0,1]\to\CP^{2\ell}$ denote the composition of $\tilde{u}:\IR\times [0,1]\to\CP^{2k}$ with the projection from $\CP^{2k}$ to $\CP^{2\ell}$. We claim that, in the case when $n\leq\ell$, that is, $u_n^0\in\CP^{2\ell}$, we indeed have that $\tilde{u}$ has image in $\CP^{2\ell}\subset\CP^{2k}$ and hence $\tilde{u}=\tilde{u}^{\ell}$. Since $\tilde{u}=u^0_n\in\CP^{2n}\subset\CP^{2\ell}$ for $T=0$, we may assume that the Floer strip $\tilde{u}$ sits in a tubular neighborhood of $\CP^{2\ell}$ in $\CP^{2k}$, possibly after passing to $0<T'<T$. It follows that we can write $\tilde{u}$ as a pair of maps, $$\tilde{u}=(\tilde{u}^{\ell}_{\perp},\tilde{u}^{\ell}):\IR\times [0,1]\to \IC^{2k-2\ell}\times \CP^{2\ell},$$ where $u^{\ell}_{\perp}$ remembers the normal component. Here $u^{\ell}_{\perp}$ shall be viewed as a section in the pull-back $(u^{\ell})^*N\to\IR\times [0,1]$ of the normal bundle $N\to\CP^{2\ell}$ of $\CP^{2\ell}\subset\CP^{2k}$ which is unitarily trivial even after applying the natural identifications. With the latter we mean the natural identification of the fibre over $(s,1)$ (= the fibre over $\tilde{u}^{\ell}(s,1)=\phi^0_1(\tilde{u}^{\ell}(s,0))$ of $N$) with the fibre over $(s,0)$ using $\phi^0_1$ for all $s\in\IR$ and, after compacifying, of the fibre over $(+\infty,t)$ with the fibre over $(-\infty,t)$ for all $t\in [0,1]$, using that $\tilde{u}^{\ell}(s,t)\to u^0_n$ as $s\to\pm\infty$. In order to see that the bundle still remains trivial, note that by varying the parameter $T>0$ we find a homotopy from $\tilde{u}=\tilde{u}^k_{n,T}$ to the constant strip $\tilde{u}^k_{n,0}=u^0_n$. \\

Now the important observation is that, since $G=G^{\ell}=G\circ\pi_{\ell}$, the projection of $\nabla G_t$ to $\IC^{2k-2\ell}$ vanishes. This however implies that the perpendicular component $\tilde{u}^{\ell}_{\perp}$ is truely holomorphic, that is, solves the unperturbed Cauchy-Riemann equation $\CR \tilde{u}^{\ell}_{\perp}=0$. Since $\tilde{u}^{\ell}_{\perp}(s,t)\to 0$ for $s\to\pm\infty$ as $u^0_n\in\CP^{2\ell}$, we can employ Liouville's theorem to show that we have $\tilde{u}^{\ell}_{\perp}=0$, that is, $\tilde{u}=\tilde{u}^{\ell}$. Note that, instead of referring to Liouville's theorem, the result can be viewed as a consequence of the minimal surface property of pseudo-holomorphic curves. Indeed, in the proof of proposition \ref{step1} below, we see that the result also follows from the fact the $L^2$-norm of $\del_s \tilde{u}^{\ell}_{\perp}$ is zero. \end{proof}

\begin{remark}\label{Liouville-2} The following observations are immediate:
\begin{itemize}
\item[i)] By the same arguments it follows that, even if we first allowed the Floer strip $\tilde{u}$ to live in the infinite-dimensional manifold $\IP(\IH)$, the finite-dimensionality of the nonlinearity ensures that it actually lives in the finite-dimensional submanifold $\CP^{2\ell}$. 
\item[ii)] Along the same lines using Liouville's theorem or the minimal surface property, it is immediate to see that, in the case of $n>\ell$, the Floer strip is constant and the fixed point $u^0_n\in\CP^{2n}$ of the free Schr\"odinger equation thus agrees with the corresponding fixed point $u^1_n$ of the nonlinear Schr\"odinger equation with convolution term. 
\item[iii)] More generally, the same argument shows that if $\hat{\psi}(n)=0$ for some $\ell<n\leq k$, then $\pi_n\circ u^{\ell}_{\perp}=0$, where $\pi_n:\IC^{2k-2\ell}\to\IC$ denotes the projection onto the $n$.th factor. Since $\psi$ is assumed to be admissible in the sense of definition \ref{admissible}, for $n\in M_{\delta}=\{n\in\IZ:\;|\exp(in^2)-1|\geq\delta\}$ the resulting fixed point $u^1_n$ of the Schr\"odinger equation with nonlinearity can\emph{not} agree with any of trivial fixed points $u^0_m$ for $m\in\IZ$ with $\hat{\psi}(m)=0$. In particular, when the nonlinearity is not zero, there is always a \emph{nontrivial} fixed point. 
\end{itemize}
\end{remark}

\section{From finite to infinite dimensions using non-standard models}

After discussing the case of finite-dimensional nonlinearities, we now want to turn to the case of general nonlinearities of convolution type. 
To this end, let us fix a nonlinearity of convolution type as in the main theorem, in particular, we assume that the underlying smoothing kernel is admissible and $|||F|||<\pi/4$. Below we show how to prove the main theorem in the general case by combining the existence result for Floer strips in finite dimensions from proposition \ref{strips} with methods from non-standard model theory. More precisely, we prove the existence of Floer strips in $\IP(\IH)$ by showing that there exists a *-finite-dimensional symplectic vector space $\IF$ that contains the symplectic Hilbert space $\IH=L^2(S^1,\IC)\cong\ell^2(\IC)$ and a *-finite-dimensional Hamiltonian flow that represent the infinite-dimensional Hamiltonian flows introduced earlier. \\

Although the appendix gives a basic introduction to non-standard model theory with details and references, we start with a quick summary of the main ideas that are needed to follow our arguments. Nonetheless we ask the reader to consult the appendix below for precise statements and further details and examples. \\

It is the fundamental idea of non-standard model theory to enlarge every \emph{standard} set $A$ to a \emph{non-standard} set $^*A$ which contains additional new ideal elements. This means that we have a map $*: V\to W$, where $V$ is a set of standard sets, called standard model, and $W$ is the corresponding set of non-standard sets, called the non-standard model. More precisely, the standard and the non-standard model come with a filtration, $V=(V_n)_{n\in\IN}$ and $W=(W_n)_{n\in\IN}$, and the map $*$ respects this filtration. In the appendix we show how to define $V=(V_n)_{n\in\IN}$ in such a way that it contains all sets that are needed for our proof. On the other hand, we also show how to prove the existence of the corresponding non-standard model $W=(W_n)_{n\in\IN}$ together with a transfer map $*: V\rightarrow W$ respecting the filtration. \\

All applications of non-standard model theory rely on the following two principles: \\
\begin{itemize}
\item\emph{Transfer principle:} If a theorem holds in the standard model $V$, then the same theorem holds in the non-standard model $W$, after replacing the elements from $V$ by their images in $W$ under the map $*$. Informally speaking, this means that every mathematical statement about standard sets also holds for their non-standard extensions. \\
\item\emph{Saturation principle:} If $(A_i)_{i\in I}$ is a collection of sets in $W$ (and $I$ is a set in $V$) satisfying $A_{i_1}\cap\ldots\cap A_{i_n}\neq\emptyset$ for all $i_1,\cdots,i_n\in I$, $n\in\IN$ \emph{(finite intersection property)}, then also the common intersection of all $A_i$, $i\in I$ is non-empty, $\bigcap_{i\in I} A_i\neq\emptyset$. \\
\end{itemize}

Denoting, as before, by $^*a\in W$ the image of $a\in V$ under $*$, the transfer principle immediately implies that  $*: V\to W$ is indeed an embedding, since $a\neq b$ in $V$ implies that $^*a\neq ^*b$ in $W$; in particular, \emph{we follow the convention to drop the star if no confusion is likely to arise. In particular, we identify every element in $\IH$ with its *-image in ${^*\IH}$.} Furthermore we have $^*\{a_1,\ldots,a_n\}=\{^*a_1,\ldots,^*a_n\}$ for all finite sets. On the other hand, if $A$ is a set in $V$ with infinitely many elements, then it easily follows from the saturation principle that ${^*}A$ is strictly larger than $A:=^*[A]:=\{^*a: a\in A\}$, see the appendix for a proof. In the latter case it holds that $A$ is not even a set in the new model $W$. In particular, this applies to our original Hilbert space $\IH(={^*}[\IH])$, viewed as a subset of ${^*}\IH$. More precisely, one can show that every non-standard model is necessarily not full, that is, there exist subsets of sets in $W$ which do not belong to $W$ itself. While this destroys all obvious logical paradoxa, on the positive side the transfer principle implies that every definable subset (subset of elements of a set in $W$ which fulfill a sentence in the language of $W$) still belongs to $W$, so the lack of fulness does not cause problems either. \\    

All applications of non-standard analysis rely on the fact that now there exist new ideal objects whose existence is an immediate consequence of the saturation principle: 
\begin{itemize}
\item[i)] There exist $r\in{^*}\IR\backslash\{0\}$ such that $|r|<1/n$ for every standard natural number $n\in\IN$. Any such $r\in{^*}\IR$ (including $r=0$) is called \emph{infinitesimal} and we write $r\approx 0$.
\item[ii)] There exist $r\in{^*}\IR$ such that $|r|>n$ for every standard natural number $n\in\IN$. Any such $r\in{^*}\IR$ is called \emph{unlimited}. Any $r\in{^*}\IR$ which is not unlimited is called \emph{limited}. 
\item[iii)] A number $r\in{^*}\IR$ is limited if and only if it is \emph{near-standard} in the sense that there exists a standard real number $s\in\IR$ with $r-s\approx 0$. For every near-standard $r\in{^*}\IR$ we call ${^{\circ}}r:=s\in\IR$ the \emph{standard part} of $r$. 
\item[iv)] Every $n\in{^*\IN}\backslash\IN$ is unlimited.  
\end{itemize}

The proof is an easy exercise, see the appendix. Using the existence of unlimited natural numbers we will now prove that the infinite-dimensional symplectic flow $\phi_t$ on $\IH$ can be represented by a *-finite-dimensional symplectic flow. For this we first need to define what we mean by a \emph{*-finite-dimensional subspace} of ${^*\IH}$. In what follows we continue to identify the Hilbert space $\IH=L^2(S^1,\IC)$ with the space $\ell^2(\IC)$ of square-summable complex-valued series $\hat{u}:\IZ\to\IC$. \\

Let $\EE(\IH)$ denote the set of finite-dimensional complex subspaces of $\IH$. Then every element in the *-image ${^*\EE(\IH)}$ of $\EE(\IH)$ is called a *-finite-dimensional complex subspace of ${^*\IH}$. Note that, since $\EE(\IH)$ is a set in the standard model $V$, it has a *-image in the non-standard model. Since every $F\in\EE(\IH)$ is a vector space over $\IC$ and a subset of $\IH$, every $F\in{^*\EE(\IH)}$ is a vector space over ${^*\IC}$ and a subset of  ${^*\IH}$ by transfer. Note that ${^*\IC}$ is indeed a field by transfer, where addition and multiplication extend the corresponding operations on $\IC$ in the sense of corollary \ref{extension}. The dimension defines a function $\dim:\EE(\IH)\to\IN$ in the standard model. By corollary \ref{extension} it follows that its *-image ${^*\dim}$ assigns to every *-finite-dimensional complex subspace $F$ a *-natural number $\dim(F):={^*\dim}(F)\in{^*\IN}$. For every $k\in\IN$ the set $$\IC^{2k+1}:=\{\hat{u}\in\IH: \hat{u}(n)=0\,\,\textrm{for all}\,\, |n|>k\}$$ is a finite-dimensional subspace in $\EE(\IH)$ of dimension $k$. By the transfer principle it follows that for every $k\in{^*\IN}$ the set $${^*\IC}^{2k+1}:=\{\hat{u}\in{^*\IH}: \hat{u}(n)=0\,\,\textrm{for all}\,\, |n|>k\}$$ is an element in ${^*\EE}(\IH)$, that is, a *-finite-dimensional subspace, of (non-standard) dimension $k\in{^*\IN}$; in particular, if $k\in{^*\IN}$ is unlimited, then it is not a finite-dimensional vector space over ${^*\IC}$. \\

In the same way as the Hilbert norm $|\cdot|$ defines a map from $\IC^k$ to $\IR^+\cup\{0\}$ for every $k\in\IN$, by transfer it follows that its *-image ${^*|\cdot|}$ defines a map from ${^*\IC}^k$ to ${^*\IR^+}\cup\{0\}$ for all $k\in{^*\IN}$. Furthermore, since for all $k\in\IN$ we have $|\hat{v}|^2=\sum_{n=-k}^{+k} |\hat{v}(n)|^2$ for $\hat{v}\in\IC^k$, for all $k\in{^*\IN}$ we have that $|\hat{v}|^2=\sum_{n=-k}^{+k} |\hat{v}(n)|^2:={^*\sum}_{n=-k}^{+k} |\hat{v}(n)|^2$ for $\hat{v}\in{^*\IC}^k$. In order to understand what a *-finite summation is, note that corresponding finite summation in $\IR$ is a function $\sum$ that assigns to every tuple $(r_{-k},\ldots,r_k)$ of elements in $\IR$ of cardinality $2k+1\in\IN$ the element $\sum_{n=-k}^k r_n:= r_{-k}+\ldots+r_k\in\IR$. By corollary \ref{extension} it follows that ${^*\sum}$  assigns to every tuple of elements in ${^*\IR}$ of (non-standard) cardinality $2k+1\in{^*\IN}$ an element in ${^*\IR}$, which we again write as $\sum_{n=-k}^k r_n$, and it has the same formal properties by the transfer principle.\\ 

For any $\hat{u},\hat{v}\in{^*\IH}$ we say that $\hat{u}$ and $\hat{v}$ are infinitesimally close to each other, $\hat{u}\approx\hat{v}$, if and only if $|\hat{u}-\hat{v}|\approx 0$, where $|\cdot|={^*|\cdot|}$ denotes the *-extension of the Hilbert norm on ${^*\IH}$. \\

\begin{proposition}\label{existence}
There exists a *-finite-dimensional complex subspace $\IF$ of ${^*\IH}$ which contains the infinite-dimensional space $\IH$ as a subspace up to an infinitesimal error, $$\IH\stackrel{\subset}{\approx}\IF\subset{^*\IH},$$ that is, for every point $u\in\IH$ there exists a point $v\in\IF$ such that $u\approx v$.
\end{proposition}

\begin{proof} We can choose $\IF={^*\IC}^{2K+1}$ for any unlimited *-natural number $K\in{^*\IN}\backslash\IN$. For this observe that for all $\hat{u}\in\IH$ and every $\eps>0$ there exists $k_0\in\IN$ such that $$\forall k\in\IN: k\geq k_0\Rightarrow d(\hat{u},\IC^{2k+1})=\min\{|\hat{u}-\hat{v}|: \hat{v}\in\IC^{2k+1}\}<\eps.$$ By transfer it follows that $$\forall k\in{^*\IN}: k\geq k_0\Rightarrow {^*d}(\hat{u},{^*\IC}^{2k+1})=\min\{|\hat{u}-\hat{v}|: \hat{v}\in{^*\IC}^{2k+1}\}<\eps.$$ Choosing any unlimited $K\in{^*\IN}\backslash\IN$, it follows from the fact that $K\geq k_0$ for any standard $k_0\in\IN$ that ${^*d}(\hat{u},{^*\IC}^{2K+1})<\eps$ for all standard $\eps>0$, that is, there exists $\hat{v}\in\IF$ with $\hat{u}\approx\hat{v}$. \end{proof}

In what follows we fix an unlimited *-natural number $N$ and define $\IF:={^*\IC}^{2N+1}$. Note that the *-extensions on $^*\IH$ of the symplectic form, the inner product and the complex structure on $\IH$ restrict to a symplectic form, inner product and complex structure on $\IF$ by corollary \ref{extension}.\\

Recall from above that every limited *-real number $r\in{^*\IR}$ is near-standard in the sense that there exists a standard real number $s\in\IR$ with $r\approx s$, called the standard part ${^{\circ}}r:=s$ of $r$. Generalizing this, we are lead to the following 

\begin{definition}\label{near-standard} An element $v\in\IF$ is called 
\begin{itemize}
\item[i)] \emph{limited} if its norm $|v|\in{^*\IR^+}\cup\{0\}$ is limited, 
\item[ii)] \emph{near-standard} if there exists $u\in\IH$ with $u\approx v$ and we call $u$ the \emph{standard part} of $u$ and write ${^{\circ}v}:=u$. 
\end{itemize}
Furthermore we call a point in $\IP(\IF)=\IS(\IF)/{^*U(1)}$ near-standard if its equivalence class is represented by a near-standard element in $\IS(\IF)\subset\IF$. \end{definition} 

Note that every point in $\IP(\IF)$ is limited by definition and the near-standardness is independent of the chosen representative in $\IS(\IF)$. Since the standard Hilbert space $\IH$ is not a set in the non-standard model, it is important to have a non-standard characterization of all near-standard points in $\IF$. It is given by the following

\begin{proposition}\label{characterization}
An element $v\in\IF$ is near-standard if and only if it is limited and for all unlimited $L\in{^*\IN}\backslash\IN$ with $L\leq N$ we have $$\sum_{n=-N}^{-L-1}|\hat{v}(n)|^2+\sum_{n=L+1}^{+N}|\hat{v}(n)|^2\;\approx\; 0.$$ In other words, if and only if $v$ has a limited norm and is infinitesimally close to all subspaces ${^*\IC}^{2L+1}\subset\IF$ of unlimited dimension $L\leq N\in{^*\IN}\backslash\IN$. 
\end{proposition} 

\begin{proof} First assume that $v\in\IF$ is limited and for all unlimited $L\in{^*\IN}\backslash\IN$ with $L\leq N$ we have $\sum_{n=-N}^{-L-1}|\hat{v}(n)|^2+\sum_{n=L+1}^{+N}|\hat{v}(n)|^2\approx 0$. Then $\hat{v}(n)\in{^*\IC}$ is limited for all $-N\leq n\leq N$ and hence near-standard. We claim that $\hat{v}\approx \hat{u}\in\IH$ where we define $\hat{u}:\IZ\to\IC$ by setting $\hat{u}(n):={^{\circ}\hat{v}(n)}\in\IC$ for all $n\in\IZ$. Fix some standard $\eps>0$. Since for all unlimited *-natural $L\leq N$ it holds that $\forall k\geq\ell\geq L: \sum_{n=-k}^{-\ell-1}|\hat{v}(n)|^2+\sum_{n=\ell+1}^{+k} |\hat{v}(n)|^2<\eps,$ by the spillover principle in proposition \ref{spillover} it then follows that there must exist a standard natural number $L\in\IN$ with the same property, in particular, $\sum_{n=-k}^{-\ell-1}|\hat{v}(n)|^2+\sum_{n=\ell+1}^{+k} |\hat{v}(n)|^2<\eps$ for all standard $k,\ell\geq L$. On the other hand, since $\hat{u}(n)\approx\hat{v}(n)$, it follows that the same inequality holds when $\hat{v}(n)$ is replaced by $\hat{u}(n)$ for all $\ell<|n|\leq k$. But with this it follows that $\hat{u}\in\ell^2(\IC)=\IH$. In order to see that $\hat{u}\approx\hat{v}$, denote $\hat{v}_{\ell}\in\IF$, $\hat{u}_{\ell}\in\IH$ by requiring that $\hat{v}_{\ell}(n)=\hat{v}(n)$, $\hat{u}_{\ell}(n)=\hat{u}(n)$ if $|n|\leq\ell$ and equal to zero else. Then it follows for all $\ell\geq L$ that $|\hat{v}-\hat{u}_{\ell}|\approx |\hat{v}-\hat{v}_{\ell}| <\eps$. \\

In the opposite direction, let us assume that $v\in\IF$ is near-standard, $v\approx u\in\IH$. Since $|u-v|\approx 0$ implies that $|v|\approx|u|<\infty$, it follows that $v$ is limited. On the other hand, we have $\hat{v}(n)\approx\hat{u}(n)$ for all $-N\leq n\leq +N$. Since $\sum_{n=-\infty}^{+\infty} |\hat{u}(n)|^2<\infty$, it follows that $\sum_{n=-\ell}^{+\ell} |\hat{u}(n)|^2\to 0$ as $\ell\to\infty$, which by proposition \ref{convergence} implies that $\sum_{n=-N}^{-L-1}|\hat{u}(n)|^2+\sum_{n=L+1}^{+N}|\hat{u}(n)|^2\approx 0$. Together with $\sum_{n=-N}^{-L-1}|\hat{v}(n)|^2\approx\sum_{n=-N}^{-L-1}|\hat{u}(n)|^2$ and $\sum_{n=L+1}^{+N}|\hat{v}(n)|^2\approx\sum_{n=L+1}^{+N}|\hat{u}(n)|^2$ the claim follows. \end{proof}

After showing that the infinite-dimensional symplectic Hilbert space $\IH$ is contained (up to an error smaller than any standard number) in a symplectic vector space which is finite-dimensional in the sense of the new model, it remains to be shown that the infinite-dimensional symplectic flow $\phi^{\IH}_t:=\phi_t$ defined by the nonlinear Schr\"odinger equation can be represented by a *-finite-dimensional symplectic flow $\phi^{\IF}_t$. With the latter we mean that for every near-standard $v\in\IP(\IF)$ we have that $\phi^{\IH}_1({^{\circ}v})={^{\circ}(\phi^{\IF}_1(v))}$; in particular, the time-one flows agree up to an error which is smaller than any positive real number, $\phi^{\IH}_1\approx\phi^{\IF}_1$. \\

Before we turn to the general case, we start with the case of the free Schr\"odinger equation. 

\begin{proposition}\label{*-free NLS} The *-image ${^*\phi^0_1}$ of the time-one map $\phi^0_1$ of the free Schr\"odinger equation is a linear symplectomorphism of $^*\IH$ which restricts to a *-finite-dimensional linear symplectomorphism on $\IF\subset{^*\IH}$. Furthermore we have ${^{\circ}({^*\phi^0_1}(v))}=\phi^0_1({^{\circ}v})$ for all near-standard $v\in\IP(\IF)$. \end{proposition}

\begin{proof} Recall from proposition \ref{free-NLS} that, after identifying $\IH=L^2(S^1,\IC)$ with $\ell^2(\IC)$, the symplectic flow map $\phi^0_1$ maps $\hat{v}\in\ell^2(\IC)$ to $\phi^0_1(\hat{v})\in\ell^2(\IC)$ given by $(\phi^0_1(\hat{v}))(n)=\exp(in^2)\cdot\hat{v}(n)$ for all $n\in\IZ$; in particular, $\phi^0_1$ naturally restricts to finite-dimensional symplectic maps on finite-dimensional complex projective spaces. Hence it follows from the transfer principle that the *-image ${^*\phi^0_1}$ restricts to a *-finite-dimensional map on all *-finite-dimensional projective subspaces ${^*\CP}^{2k}\subset{^*\IP(\IH)}=\IP({^*\IH})$, $k\in{^*\IN}$. For the last statement, observe that it follows from the fact that $\phi^0_1:\IH\to\IH$ preserves the norm that $u\approx v$ implies ${^*\phi^0_1}(v)\approx{^*\phi^0_1}(u)=\phi^0_1(u)$. \end{proof}

For the general case everything furthermore relies on a finite-dimensional approximation result of the flow of $G_t$ ($F_t$), where we now crucially make use of the special form of the nonlinearity. To this end, define for the given admissible convolution kernel $\psi$ with Fourier series expansion $\psi(x)=(2\pi)^{-1/2}\sum_{n=-\infty}^{+\infty}\hat{\psi}(n)\exp(inx)$ for each $k\in\IN$ the approximating kernel $$\psi^k(x)\,:=\,\frac{1}{\sqrt{2\pi}}\sum_{n=-k}^k \hat{\psi}(n)\exp(inx)$$ and define the resulting sequence of Hamiltonians $G^k_t$ ($F^k_t$) by $$G^k_t\,:=\,F^k_t\;\circ\;\phi^0_{-t}\,\,\textrm{with}\,\,F^k_t(u)\;:=\;\frac{1}{2}\int_0^{2\pi} f(|(u*\psi^k)(x)|^2,x,t)\;dx$$ for all $k\in\IN$. Note that $F^k_t$ then defines a finite-dimensional nonlinearity in the sense of proposition \ref{finite-dim}. 

\begin{lemma}\label{approx} For each $k\in\IN$ the Hamiltonian flow $\phi^{G,k}_t$ of the Hamiltonian $G^k_t$ restricts  a finite-dimensional Hamiltonian flow on $\CP^{2k}\subset\IP(\IH)$. Furthermore the sequence of time-dependent Hamiltonians $G^k_t$ converges uniformly on $\IP(\IH)$ with all derivatives to the original Hamiltonian $G_t$ as $k\to\infty$. In particular, the same holds true for the symplectic time-one maps $\phi^{G,k}_1$ and $\phi^G_1$, and the Hofer norm $|||G|||$ of $G_t:\IP(\IH)\to\IR$ is finite. \end{lemma}

\begin{proof} Based on the fact that the flow $\phi^0_t$ restricts to finite-dimensional flows by proposition \ref{free-NLS}, for the first statement it suffices to observe that the symplectic gradient $X^{F,k}_t$ of $F^k_t:\IH\to\IR$ given by $$X^{F,k}_t(u)=\del_1 f(|(u*\psi^k)(x)|^2,x,t)(u*\psi^k)*\psi^k$$ has vanishing Fourier coefficients, $\widehat{X^{F,k}_t(u)}(n)=0$, for $|n|>k$. Since the supremum norm of $u*\psi - u*\psi^k$ can be bounded by $$\|u*\psi - u*\psi^k\|_{\infty}\leq \|u\|_2\cdot\|\psi-\psi^k\|_2,$$ it follows from $\|\psi-\psi^k\|_2\to 0$ as $k\to\infty$ and $\|u\|_2=1$ for all $u\in \IS(\IH)$ that $u*\psi^k\to u*\psi$ uniformly as $k\to\infty$. On the other hand, since $f$ is assumed to be smooth, it immediately follows that $F^k_t(u)\to F_t(u)$ and hence $G^k_t(u)\to G_t(u)$ as $k\to\infty$, uniformly with all derivatives. \end{proof}

In order to see that this lemma immediately leads to the existence of the desired *-finite-dimensional symplectic flow $\phi^{\IF}_t$, we first observe that, by corollary \ref{extension},  the *-image of the sequence $(G^k_t)_{k\in\IN}$ provides us with a sequence of Hamiltonians ${^* G}^k_t:{^*\IH}\to{^*\IR}$ over all non-standard natural numbers $k\in{^*\IN}$ with ${^*G}^k_t(u)=G^k_t(u)$ if $k\in\IN\subset{^*\IN}$ and $u\in\IH\subset{^*\IH}$ are both standard. In the same way, the *-image of the corresponding sequence $(\phi^{G,k})_{k\in\IN}$ of finite-dimensional Hamiltonian flows defined by $\del_t\phi^{G,k}_t=X^{G,k}_t\circ\phi^{G,k}_t$ provides us with a sequence of *-finite-dimensional Hamiltonian flows ${^*\phi^{G,k}}:{^*\IR}\times{^*\IH}\to {^*\IH}$ over all non-standard natural numbers $k\in{^*\IN}$. 

\begin{remark}\label{differentiable} In order to explain the relation between both extensions, we quickly need to review the notion of differentiability in the non-standard sense. Since for every $k\in\IN$ the Hamiltonian flow map $\phi^{G,k}:\IR\times\IH\to\IH$  is differentiable with respect to $t\in\IR$ with $\del_t \phi^{G,k}_t=X^{G,k}_t\circ\phi^{G,k}_t$ where $X^{G,k}=i\nabla G^k$ is the symplectic gradient of $G^k$, it follows from the transfer principle that for every $k\in{^*\IN}$ the map ${^*\phi^{G,k}}:{^*\IR}\times {^*\IH}\to{^*\IH}$ is differentiable with respect to $t\in{^*\IR}$ in the non-standard sense and $\del_t {^*\phi^{G,k}_t}={^*X^{G,k}_t}\circ{^*\phi^{G,k}_t}$. This means that for every $(t,u)\in{^*\IR}\times{^*\IH}$ and every $\eps\in{^*\IR}^+$ there exists $\delta\in{^*\IR}^+$ such that for all $t'\in{^*\IR}$ with $|t'-t|<\delta$ we have $$\frac{|{^*\phi^{G,k}_{t'}}(u)-{^*\phi^{G,k}_t}(u)-{^*X^{G,k}_t}({^*\phi^{G,k}_t}(u))\cdot (t'-t)|}{|t'-t|}<\eps.$$ Note that for every $k\in{^*\IN}$ the symplectic gradient ${^*X^{G,k}}$ can be obtained by taking the *-extension of the sequence $(X^{G,k})_{k\in\IN}$ or be defined using the non-standard differential of $G^k$; the transfer principle ensures that in both cases one obtains the same result. \end{remark}
 
Setting $G^{\IF}:={^*G}^N$ and $\phi^{\IF}_t={^*\phi^{G,N}_t}$ for $\dim\IF=2N+1$ as well as $G^{\IH}:=G$ and $\phi^{\IH}_t=\phi_t$, using the transfer principle (in particular its consequences \ref{convergence} and \ref{extension}) the above lemma implies the proof of the following 

\begin{proposition}\label{approximation} Assume that $v\in\IF$ is nearstandard with $v\approx u\in\IH$. Then it holds that $G^{\IF}(v)\approx G^{\IH}(u)$ and $\nabla G^{\IF}(v)\approx \nabla G^{\IH}(u)$, so that $\phi^{\IF}_1(v)\approx \phi^{\IH}_1(u)$ and $|||G^{\IF}|||\approx |||G^{\IH}|||$. In particular, the infinite-dimensional flow $\phi^{\IH}_t$ can be represented without loss of information by the *-finite-dimensional flow $\phi^{\IF}_t$, $$\phi^{\IH}_1({^{\circ}v})\,=\,{^{\circ}(\phi^{\IF}_1(v))}\,\,\textrm{for all near-standard}\,\, v\in\IP(\IF).$$ \end{proposition}

\begin{proof} Here ${^{\circ}v}\in\IH$ denotes the standard part of the near-standard element $v\in\IF$ in the sense of definition \ref{near-standard}. The crucial observation is that, by proposition \ref{convergence}, it follows from the above lemma that for all \emph{unlimited} $K\in{^*\IN}\backslash\IN$ we have $G^K:={^*G}^K\approx{^*G}$ as well as $\nabla G^K\approx {^*\nabla G}$ on ${^*\IP(\IH)}=\IP({^*\IH})$. Note that this implies the statement about the Hofer norm as well as for the flow, $\phi^{G,\IF}_1=\phi^{{^*G},N}_1\approx\phi^{^*G}_1={^*(\phi^{G,\IH}_1)}$. For the last statement one could alternatively use that $\phi^{G,k}_t\to\phi^G_t$ as $k\to\infty$ and again apply proposition \ref{convergence}. In order to finish the proof it remains to observe that, as $G$ is continuous (in the standard way), for every $u\in\IH$ and every $\eps\in\IR^+$ there exists some $\delta\in\IR^+$ such that $$\forall v\in\IH: |u-v|<\delta\implies |G(u)-G(v)|<\eps.$$ Applying the transfer principle only to the last statement, it follows that $$\forall v\in{^*\IH}: |u-v|<\delta\implies |{^*G}(u)-{^*G}(v)|<\eps$$ with ${^*G}(u)=G(u)$ due to proposition \ref{extension}. If $v\in\IF\subset{^*\IH}$ satisfies $u\approx v$, that is, $|u-v|<\delta$ for all $\delta\in{^*\IR}^+$, we hence get that $|G(u)-{^*G}(v)|<\eps$ for all $\eps\in\IR^+$, that is, $$G^{\IH}(u)\approx {^*G}(v)\approx G^{\IF}(v);$$ for the gradient and the flow map the argument is the same.  \end{proof}

After showing how the infinite-dimensional flow defined by the nonlinear Schr\"odinger equation can be represented by symplectic flow which is finite-dimensional in the sense of the non-standard model, and collecting the key results from finite-dimensional Floer theory, we now show how the latter results can be elegantly used to prove our main theorem. The key tool in order to achieve this is the afore-mentioned transfer principle of non-standard model theory which states that every *-finite-dimensional object can be treated like a finite-dimensional object. In particular, by applying the transfer principle to proposition \ref{strips}, without any extra work we can immediately establish the existence of Floer strips in the *-finite-dimensional projective space $\IP(\IF)={^*\CP}^{2N}$ for the Hamiltonian $G^{\IF}_t={^*G}^N_t:\IP(\IF)\to{^*\IR}$. 

\begin{corollary}\label{*-strips} Fix some unlimited $T\in{^*\IR^+}$. For every $n\in\IN$ there exists a Floer strip $\tilde{u}^{\IF}_n:{^*\IR}\times {^*[0,1]}\to{^*\CP}^{2N}=\IP(\IF)$ satisfying the periodicity condition $\tilde{u}^{\IF}_n(\cdot,1)={^*\phi^0_1}(\tilde{u}^{\IF}_n(\cdot,0))$, the asymptotic condition $\tilde{u}^{\IF}_n(s,\cdot)\to u^0_n$ as $s\to\pm\infty$, and the perturbed Cauchy-Riemann equation $$0=\CR^T_G\tilde{u}^{\IF}_n=\CR\tilde{u}^{\IF}_n - \varphi_T(s)\cdot \nabla G^{\IF}_t(\tilde{u}^{\IF}_n).$$ Furthermore, the energy $E(\tilde{u}^{\IF}_n)$ is bounded by $2|||G^{\IF}|||<\pi/2$. \end{corollary}

\begin{proof} Since we know that for every $k\in\IN$, $n\in\IN$, $T\in\IR^+$ there exists a Floer strip $\tilde{u}^k_{n,T}:\IR\times [0,1]\to\CP^{2k}$, it follows from the transfer principle that for every $k\in{^*\IN}$, $n\in{^*\IN}$ and $T\in{^*\IR^+}$ there exists a map $\tilde{u}^k_{n,T}: {^*\IR}\times {^*[0,1]}\to{^*\CP}^{2k}$ satisfying the corresponding properties in the non-standard sense, see the discussion below. Note that for every $k\in{^*\IN}$, $n\in{^*\IN}$ and $T\in{^*\IR^+}$ the tuple $(k,n,T,\tilde{u}^k_{n,T})$ is an element in the *-extension ${^*\IM}$ of the universal moduli space $\IM$, defined as the set of all tuples $(k,n,T,\tilde{u})$, where $k\in\IN$, $n\in\IN$, $T\in\IR^+$ and $\tilde{u}=\tilde{u}^k_{n,T}$ is a Floer strip as in proposition \ref{strips}. \end{proof}

\begin{remark}\label{*-differentiable} We emphasize that all the claimed properties of the map $\tilde{u}^{\IF}_n:{^*\IR}\times {^*[0,1]}\to\IP(\IF)$ are to be understood in the \emph{non-standard sense}:
\begin{itemize}
\item[i)] The asymptotic condition $\tilde{u}^{\IF}_n(s,\cdot)\to u^0_n$ as $s\to\pm\infty$ means that for all $\eps\in{^*\IR}^+$ there exists $R\in{^*\IR^+}$ such that $d(\tilde{u}^{\IF}_n(s,t),u^0_n)<\eps$ if $|s|>R$. In particular, even for standard $\eps$ the corresponding $R$ will in general be an unlimited *-real number. 
\item[ii)] For the validity of the perturbed Floer equation $\CR\tilde{u}^{\IF}_n - \varphi_T(s)\cdot \nabla G^{\IF}_t(\tilde{u}^{\IF}_n)=0$ with $\CR\tilde{u}^{\IF}=\del_s\tilde{u}^{\IF}+i\del_t\tilde{u}^{\IF}$ one has to observe that all derivatives are to be considered in the non-standard sense. In order to make this fully explicit, let us observe that, just as in the standard case, $\IP(\IF)={^*\CP^{2N}}$ can be covered by $2N+1$ natural coordinate charts $\varphi_m:{^*\IC}^{2N}\to{^*\CP}^{2N}$, $0\leq m\leq 2N$ by setting $z_m:=1$ in $[z_0:\ldots:z_{2N}]\in{^*\CP}^{2N}$. By transfer we know that for every $(s,t)\in{^*\IR}\times{^*[0,1]}$ there exists $0\leq m\leq 2N$ and $r\in{^*\IR^+}$ such that $|(s',t')-(s,t)|<r$ implies $\tilde{u}(s',t')\in\varphi_m({^*\IC}^{2N})\subset{^*\CP}^{2N}$. Then $\tilde{u}=\tilde{u}^{\IF}_n$ is differentiable in the non-standard sense at $(s,t)\in{^*\IR}\times{^*[0,1]}$ if and only if $\tilde{u}_m:=\varphi_m^{-1}\circ\tilde{u}: {^*B}^2_r(s,t)\to{^*\IC}^{2N}$ has this property, where ${^*B}^2_r(s,t)=\{(s',t')\in{^*\IR}\times{^*[0,1]}: |(s',t')-(s,t)|<r\}$. The derivatives $\del_s\tilde{u}_m(s,t)$, $\del_t\tilde{u}_m(s,t)\in{^*\IC}^{2N}$ are then characterized as follows: for all $\eps\in{^*\IR}^+$ there exists $\delta(\leq r)\in{^*\IR}^+$ such that for all $(s',t')\in{^*B}^2_{\delta}(s,t)$ we have 
\begin{eqnarray*}
\Big|\frac{\tilde{u}_m(s',t)-\tilde{u}_m(s,t)}{s'-s}\;-\;\del_s\tilde{u}_m(s,t)\Big|<\eps,\\
\Big|\frac{\tilde{u}_m(s,t')-\tilde{u}_m(s,t)}{t'-t}\;-\;\del_t\tilde{u}_m(s,t)\Big|<\eps.
\end{eqnarray*}
\item[iii)] In the same way as the energy of a Floer strip is a map $E:\IM\to\IR^+\cup\{0\}$ which assigns to every $(k,n,T,\tilde{u})\in\IM$ the energy $E(\tilde{u})$ defined in proposition \ref{strips}, its *-extension ${^*E}$ assigns to every $(k,n,T,\tilde{u})\in{^*\IM}$ a non-negative *-real number $E(\tilde{u}):={^*E}(k,n,T,\tilde{u})$. 
\end{itemize}\end{remark}

\section{Floer strips in infinite-dimensional projective spaces}

In this section we prove the main theorem. While we have already established that, for every $n\in\IN$ and every $T\in{^*\IR}^+$, there exists a non-standard map $\tilde{u}=\tilde{u}^{\IF}_{n,T}:{^*\IR}\times {^*[0,1]}\to\IP(\IF)$, note that this does not immediately imply the existence of a Floer strip in $\IP(\IH)$. Furthermore, even when we would know that the image of $\tilde{u}$ would be contained in $\IP(\IH)$, note that the asymptotic condition and the first derivatives appearing in the Cauchy-Riemann equation are only to be understood in the non-standard sense. In particular, the derivatives of the *-smooth map $\tilde{u}$ could a priori have unlimited norm at every point. Applying the transfer principle to well-known fundamental properties of finite-dimensional holomorphic curves, as first step we however can prove   

\begin{proposition}\label{no-bubbling} For every $n\in\IN$ and $T\in{^*\IR^+}\cup\{0\}$ and every $\ell\in\IN$ the $\ell$.th derivative of the map $\tilde{u}=\tilde{u}^{\IF}_{n,T}:{^*\IR}\times {^*[0,1]}\to\IP(\IF)$ has a limited norm at every point $(s,t)\in{^*\IR}\times{^*[0,1]}$. \end{proposition}

For the proof we use a non-standard version of the classical bubbling-off argument from (\cite{MDSa}, chapter 4) together with elliptic regularity from (\cite{MDSa}, appendix B). Apart from the fact that the a priori estimate for bubbling, the elliptic estimate used to bound higher Sobolev norms as well as the Sobolev embedding theorems have analogues for non-standard maps to the *-finite-dimensional projective space $\IP(\IF)$ by the transfer principle, the crucial observation for the proof is that the constants appearing in the used inequalities are still limited numbers. The proof crucially relies on the following 

\begin{lemma} At all points $(s,t)\in{^*\IR}\times{^*[0,1]}$ the first derivatives $\del_s\tilde{u}(s,t)$, $\del_t\tilde{u}(s,t)$ of $\tilde{u}=\tilde{u}^{\IF}$ are limited. \end{lemma}

Here we say that $\del_s\tilde{u}(s,t)$, $\del_t\tilde{u}(s,t)$ of $\tilde{u}$ are limited if and only if they have an limited norm, where norm refers to the non-standard version of Riemannian metric on $\IP(\IF)={^*\CP}^{2N}$ which exists by transfer. In order to avoid working with non-standard versions of Riemannian metrics on manifolds, note that we can alternatively directly work in the local coordinate charts $\varphi_m:{^*\IC}^{2N}\to{^*\CP}^{2N}$, $0\leq m\leq 2N$ from remark \ref{*-differentiable}. Then the statement is equivalent to requiring that the first derivatives $\del_s\tilde{u}_m(s,t)$,  $\del_t\tilde{u}_m(s,t)$ of $\tilde{u}_m:=\varphi_m^{-1}\circ\tilde{u}$ are limited with respect to the standard norm $|(z_1,\ldots,z_{2N})|=\sum_{n=1}^{2N} |z_n|^2\in{^*\IR^+}\cup\{0\}$ on ${^*\IC}^{2N}$, see the last section. 

\begin{proof} For the proof we use a non-standard version of the classical bubbling-off argument from (\cite{MDSa}, chapter 4). In order to show that the supremum norm of $\del_s\tilde{u}$ is limited, we are now essentially going to use that the energy $E(\tilde{u})$ of $\tilde{u}$ is strictly smaller than the minimal energy of a holomorphic sphere in $\IP(\IF)$. Note that, since $\CR\tilde{u}-\varphi_T(s)\nabla G^{\IF}_t(\tilde{u})=0$ and $\nabla G^{\IF}$ is limited due to proposition \ref{approximation}, the latter implies that also the supremum norm of $\del_t\tilde{u}$ is limited\\

To the contrary, assume that $\max\{|\del_s\tilde{u}(z)|:z\in{^*\IR}\times{^*[0,1]}\}=C$ is an unlimited *-real number and choose $z_0\in{^*S}^2$ such that $|\del_s\tilde{u}(z_0)|=C$. Note that by transfer the maximum exists, due to the asymptotic condition and we assume without loss of generality that $z_0=(0,1/2)$. As in the classical bubbling-off proof we define $\tilde{v}:{^*B}^2_{\sqrt{C}}(0)\to\IP(\IF)$ by $\tilde{v}(z):=\tilde{u}(z/C+z_0)$, such that $|\del_s\tilde{v}(0)|=1$ and $|\del_s\tilde{v}(z)|\leq 1$ for all $z\in{^*B}^2_{\sqrt{C}}(0)$.  Because $C\in{^*\IR}^+$ was assumed to be unlimited, note that ${^*B}^2_{\sqrt{C}}(0)\subset{^*\IC}$ is a disk of unlimited radius; in particular, it contains the full complex plane $\IC$ as a subset. For each $r\in[0,C]\subset{^*\IR}^+$ define $\gamma_r: {^*S}^1\to\IP(\IF)$ by $\gamma_r(\theta):=v(re^{i\theta})$. For every $k\in\IN$ denote by $\ell_k$ the map which assigns to each loop $\gamma:S^1\to\CP^{2k}$ its length with respect to the canonical Riemannian metric and $E_{\omega}(v):=\int v^*\omega$ the symplectic area of a disk map $v: B^2_r(0)\to\CP^{2k}$. By taking the *-extension of the resulting sequence $(\ell_k)_{k\in\IN}$, it follows that there also exists a map $\ell:=\ell_N$ which assigns to every loop $\gamma_r:{^*S}^1\to\IP(\IF)$ its length with respect to the *-Riemannian metric on $\IP(\IF)={^*\CP}^{2N}$; in the same way the symplectic area $E_{\omega}(\tilde{v})$ of a non-standard map $\tilde{v}: {^*B}^2_r(0)\to\IP(\IF)={^*\CP}^{2N}$ is defined using the *-extension. Writing $r\lessapprox s$ if $r<s$ or $r\approx s$, we can formulate the following  \\

\noindent\emph{Claim: There exists some $\rho\in[\sqrt{C}/2,\sqrt{C}]$ such that $\ell(\gamma_{\rho})\approx 0$ and for the symplectic area of the restricted map $\tilde{v}_{\rho}=\tilde{v}:{^*B}^2_{\rho}(0)\to\IP(\IF)$ we have $E_{\omega}(v_{\rho})\leq 2|||G|||<\pi$. Furthermore we have the a priori estimate $|\del_s\tilde{v}(0)|^2\lessapprox E_{\omega}(v_{\rho})/(\rho^2\pi)$.}\\

We prove this claim by proving the corresponding standard result. Fix $k\in\IN$. By abuse of notation, let us denote by $\tilde{u}:\IR\times [0,1]\to\CP^{2k}$ a Floer strip as in proposition \ref{strips} and define $\tilde{v}$ as above; further let $C>0$ denote any positive real number. Setting $\tilde{w}(r,\theta)=\tilde{v}(r e^{i\theta})$ and using the finiteness of the $C^1$-norm of $G$, an easy computation shows that $$E_{\omega}(\tilde{v})\,-\,\int_{B^2_{1/\sqrt{C}}(z_0)} |\del_s\tilde{u}|^2+|\del_t\tilde{u}-\varphi_T(s)\nabla G_t(\tilde{u})|^2\;ds\;dt\,\to\, 0$$ and hence $$\int_0^{\sqrt{C}}\int_0^{2\pi} |\del_{\theta}\tilde{w}|^2\;rd\theta\;dr\,-\,\int_{B^2_{1/\sqrt{C}}(z_0)} |\del_s\tilde{u}|^2+|\del_t\tilde{u}-\varphi_T(s)\nabla G_t(\tilde{u})|^2\;ds\;dt$$ converges to $0$ as $C\to\infty$. \\

Together with $E(\tilde{u})\leq 2|||G|||$ and Cauchy-Schwarz, this implies that $$(2\pi)^{-1}\cdot \int_{\sqrt{C}/2}^{\sqrt{C}}\Big(\int_0^{2\pi}|\del_{\theta}\tilde{w}|d\theta\Big)^2\;dr\,\leq\,\int_{\sqrt{C}/2}^{\sqrt{C}}\int_0^{2\pi} |\del_{\theta}\tilde{w}|^2\;d\theta\;rdr<\pi$$ for $C>0$ sufficiently large. In particular, by setting $$\ell^k_{\min}:=\min\{\ell(\gamma_r):\;r\in [\sqrt{C}/2,\sqrt{C}]\},$$ it follows that $\ell^k_{\min}\leq \sqrt{2\pi^2/(\sqrt{C}/2))}\to 0$ as $C\to\infty$; in other words, for every $\eps>0$ there exists $C_0>0$ such that $\ell^k_{\min}<\eps$ if $C\geq C_0$. Since $C_0$ can be chosen to be independent of $k\in\IN$ and every unlimited $C\in{^*\IR}^+$ is greater than any standard $C_0$, it follows from transfer that $\ell^N_{min}\approx 0$. In order to finish the proof of the claim, for the result on the symplectic area it suffices to observe that by the first limit we get $E_{\omega}(\tilde{v})\lessapprox 2|||G|||<\pi$ by employing corollary \ref{*-strips}. And finally, for the a priori estimate, we just need to observe that $$\CR\tilde{v}\,=\,-C^{-1}\varphi_T(s)\nabla G_t(\tilde{v})\,\to\, 0\,\,\textrm{as}\,\,C\to\infty,$$ so that the result follows as in step 2 in the proof of proposition \ref{step1} by replacing $\nabla^{\ell}_{\perp} G^k$ by $C^{-1}\varphi_T(s)\nabla G_t$ (and $\tilde{u}^{\ell}_{\perp}$ by $\tilde{v}$). \\

In order to finish the proof of the lemma we observe that, due to the fact that $\gamma_{\rho}$ has infinitesimal length, there exists a filling disk $\tilde{\gamma}_{\rho}:{^*B}^2_{1/\rho}(0)\to\IP(\IF)$ with $E_{\omega}(\tilde{\gamma}_{\rho})\approx 0$. This is a consequence of the fact that every sufficiently small loop $\gamma: S^1\to\CP^{2k}$ has a unique local filling $\tilde{\gamma}:B^2_1(0)\to\CP^{2k}$, that is, there exist $\ell_{\max}>0$ and $c>0$ such that $E_{\omega}(\tilde{\gamma})\leq c\ell(\gamma)^2$ if $\ell(\gamma)\leq\ell_{\max}$, see the proof of lemma 4.2.3 in \cite{MDSa}. Since $\tilde{v}_{\rho}$ and $\tilde{\gamma}_{\rho}$ match on their boundaries, it follows from transfer that $E_{\omega}(\tilde{v}_{\rho})+E_{\omega}(\tilde{\gamma}_{\rho})=m\pi$ for some $m\in{^*\IN}$. But since $E_{\omega}(\tilde{v}_{\rho})+E_{\omega}(\tilde{\gamma}_{\rho})<\pi$, it follows that $m=0$, in particular, $E_{\omega}(\tilde{v}_{\rho})\approx 0$. Applying now the a priori estimate $|\del_s\tilde{v}(0)|^2\lessapprox E_{\omega}(v_{\rho})/(\rho^2\pi)$, it follows that $\del_s\tilde{v}(0)\approx 0$ - in contradiction to $|\del_s\tilde{v}(0)|=1$. \end{proof}

\begin{proof}\emph{(of the proposition)} With the help of the above lemma, we can now give the proof of proposition \ref{Floer-smooth}. As for the definition of the differential in the non-standard context in remark \ref{*-differentiable}, we are going to make use of the fact that $\IP(\IF)={^*\CP^{2N}}$ can be covered by $2N+1$ natural coordinate charts $\varphi_m:{^*\IC}^{2N}\to{^*\CP}^{2N}$, $0\leq m\leq 2N$ by setting $z_m:=1$ in $[z_0:\ldots:z_{2N}]\in{^*\CP}^{2N}$. Fix $(s,t)\in{^*\IR}\times{^*[0,1]}$. By transfer there exists $0\leq m\leq 2N$ and $r\in{^*\IR^+}$ with $r\leq 1$ such that $|(s',t')-(s,t)|<r$ implies $\tilde{u}(s,t)\in\varphi_m({^*\IC}^{2N})\subset {^*\CP}^{2N}$, and we define $\tilde{u}_m:=\varphi_m^{-1}\circ\tilde{u}: {^*B}^2_r(s,t)\to{^*\IC}^{2N}$. Note that, since the first derivatives of $\tilde{u}$ are limited, the radius $r$ can indeed be chosen to be a standard positive number. For the proof we have to show that the $C^{\ell}$-norm of $\tilde{u}_m$ is limited for all standard $\ell\in\IN$.\\

Note that, since the $C^{\ell}$-norm $\|\cdot\|_{C^{\ell}}$ is a map which assigns to every $\ell$-times differentiable map from a closed two-dimensional ball to $\IC^{2k}$ an element in $\IR^+\cup\{0\}$, it follows that its *-extension ${^*\|\cdot\|_{C^{\ell}}}$ assigns to $\tilde{u}_m:{^*B}^2_r(s,t)\to{^*\IC}^{2N}$ a non-negative *-real number. For this we use that $\tilde{u}$ and hence $\tilde{u}_m$ is smooth in the non-standard sense; furthermore we know that ${^*\|\tilde{u}_m\|_{C^{\ell}}}\in{^*\IR^+}$. Since from the lemma we know that the maximum norms of $\del_s\tilde{u}_m$ and $\del_t\tilde{u}_m$ are limited numbers and the maximum is attained, we already know that ${^*\|\tilde{u}_m\|_{C^1}}$ is limited. In order to show that ${^*\|\tilde{u}_m\|_{C^{\ell}}}$ is a limited number for all $\ell\in\IN$, we apply the transfer principle to the classical elliptic regularity result, together with proposition \ref{approximation}. For this we fix some standard $p>2$ and introduce for every standard $\ell\geq 1$ the *-extension of the Sobolev $H^{\ell,p}$-norm ${^*\|\cdot\|_{\ell,p}}={^*\|\cdot\|_{H^{\ell,p}}}$ which by transfer assigns to $\tilde{u}_m$ a non-negative *-real number ${^*\|\tilde{u}_m\|_{\ell,p}}$. By applying the transfer principle to the well-known Sobolev embedding theorem relating the Sobolev $H^{\ell,p}$-norms with the $C^{\ell}$-norms for different $\ell\in\IN$, note that for all $\ell'\leq\ell-2/p$ we have $${^*\|\tilde{u}_m\|_{C^{\ell'}}}\leq c_0\cdot{^*\|\tilde{u}_m\|_{H^{\ell,p}}}\,\,\textrm{with a standard constant}\,\,c_0\in\IR^+.$$ For the latter we use that the constant $c_0=c_0(\ell,p)$ is independent of the dimension of the target space. \\

We now prove by induction that ${^*\|\tilde{u}_m\|_{\ell,p}}$ is a limited number for all standard $\ell\geq 1$. For the induction start, note that it follows from transfer and $r\leq 1$ that $${^*\|\tilde{u}_m\|_{H^{1,p}}}\leq \pi^{1/p}\cdot {^*\|\tilde{u}_m\|_{C^1}}$$ is limited. Note that here and in what follows we use the computation rules from proposition \ref{limited-infinitesimal}. For the induction step, let us assume that ${^*\|\tilde{u}_m\|_{\ell,p}}$ is limited. Note that $\CR\tilde{u}=\varphi_T(s)\cdot\nabla G^{\IF}_t(\tilde{u})$ is equivalent to  $\CR\tilde{u}_m=\eta$ with $\eta=\varphi_T(s)\cdot (\varphi_m^*\nabla G^{\IF}_t)(\tilde{u}_m)$; in particular ${^*\|\eta\|_{\ell,p}}$ is limited if and only if the $H^{\ell,p}$-norm of $\nabla G^{\IF}_t(\tilde{u})$ is limited. Since by lemma \ref{approx} we have for all $\ell\in\IN$ that $\|\nabla G^k\|_{C^{\ell}}\to\|\nabla G\|_{C^{\ell}}$ as $k\to\infty$, it follows from proposition \ref{convergence} that ${^*\|\nabla G^{\IF}\|_{C^{\ell}}}\approx\|\nabla G\|_{C^{\ell}}\in\IR^+\cup\{0\}$ is limited.  By applying the transfer principle to the second inequality in (\cite{MDSa}, proposition B.1.7) we have $${^*\|\nabla G^{\IF}_t(\tilde{u})\|_{H^{\ell,p}}}\leq c_1({^*\|\nabla G^{\IF}_t\|_{C^{\ell}}}+1){^*\|\tilde{u}\|_{H^{\ell,p}}}$$ with a standard and hence limited constant $c_1\in\IR^+$; for the latter we again use that the constant in (\cite{MDSa}, proposition B.1.7) is independent of the dimension of the target space. Since ${^*\|\tilde{u}\|_{\ell,p}}$ is limited, it follows that the $H^{\ell,p}$-norm of $\nabla G^{\IF}_t(\tilde{u})$ and hence ${^*\|\eta\|_{\ell,p}}$ is limited. In order to complete the induction step, we apply the transfer principle to the local regularity for the $\CR$-operator in (\cite{MDSa}, theorem B.3.4) in order to obtain $${^*\|\tilde{u}_m\|_{\ell+1,p}}\leq c_2 \big({^*\|\CR\tilde{u}_m\|_{\ell,p}}+{^*\|\tilde{u}_m\|_p}\big).$$ Since the constant $c_2$ in (\cite{MDSa}, theorem B.3.4) is again independent of the dimension of the target space, it follows that we can still use the same standard $c_2\in\IR^+$, which together with the limitedness of ${^*\|\CR\tilde{u}_m\|_{\ell,p}}={^*\|\eta\|_{\ell,p}}$ and ${^*\|\tilde{u}_m\|_p}\leq{^*\|\tilde{u}_m\|_{\ell,p}}$ proves that ${^*\|\tilde{u}_m\|_{\ell+1,p}}$ is still limited. \end{proof}

As the next step we show that the non-standard Floer strip $\tilde{u}=\tilde{u}^{\IF}_n:{^*\IR}\times {^*[0,1]}\to\IP(\IF)$ from proposition \ref{*-strips} is near-standard in the sense that for all $(s,t)\in{^*\IR}\times{^*[0,1]}$ the point $\tilde{u}(s,t)\in\IP(\IF)$ is near-standard in the sense of definition \ref{near-standard}. In particular, after applying the standard part map from definition \ref{near-standard} and defining $({^{\circ}\tilde{u}})(s,t):={^{\circ}(\tilde{u}(s,t))}$, we obtain a map $\tilde{u}^{\IH}_n:={^{\circ}\tilde{u}}:\IR\times [0,1]\to\IP(\IH)$. We want to emphasize that our near-standardness proof uses the limitedness of the non-standard derivatives of $\tilde{u}$, that is, it relies on a bubbling-off argument.   
 
\begin{proposition}\label{step1} For every $n\in\IN$ the map $\tilde{u}^{\IF}_n:{^*\IR}\times {^*[0,1]}\to\IP(\IF)$ is near-standard. \end{proposition}

As with our last proposition we will prove this proposition by combining non-standard results which are obtained by applying the transfer principle to standard results from finite dimensions. For some standard $T>0$ and $n\in\IN$ as well as $k\geq\ell\geq n$ let $\tilde{u}^k_n=\tilde{u}^k_{n,T}:\IR\times [0,1]\to\CP^{2k}$ be a Floer strip as in proposition \ref{strips} and consider $\CP^{2\ell}\subset\CP^{2k}$. As in the discussion about the case of finite-dimensional nonlinearities, we know that, for $T>0$ sufficiently small, the Floer strip $\tilde{u}=\tilde{u}^k_{n,T}$ sits in a tubular neighborhood of $\CP^{2\ell}$ in $\CP^{2k}$. Denoting by $\tilde{u}^{\ell}=\pi_{\ell}\circ\tilde{u}$ the canonical projection of the Floer strip onto $\CP^{2\ell}\subset\CP^{2k}$, we claim that we can again write $\tilde{u}$ as a pair of maps, $$\tilde{u}=(\tilde{u}^{\ell}_{\perp},\tilde{u}^{\ell}):\IR\times [0,1]\to\IC^{2k-2\ell}\times \CP^{2\ell},$$ where $u^{\ell}_{\perp}$ remembers the normal component. More precisely, $u^{\ell}_{\perp}$ shall be viewed as a section in the pull-back $(u^{\ell})^*N\to\IR\times [0,1]$ of the normal bundle of $\CP^{2\ell}\subset\CP^{2k}$ which is unitarily trivial even after applying the natural identifications, see the discussion about finite-dimensional nonlinearities. \\

The proof of proposition \ref{step1} relies on non-standard versions of the following lemmata.

\begin{lemma}\label{L^2-Hofer} The $L^2$-norm of $\del_s \tilde{u}^{\ell}_{\perp}$ can be bounded from above in terms of the Hofer norm, $$\|\del_s \tilde{u}^{\ell}_{\perp}\|_2\,\leq\,2\;||| G^k - G^{\ell} |||.$$ \end{lemma} 

\begin{proof} The proof of this lemma builds on lemma 8.1.6, remark 8.1.7 and the proof of theorem 9.1.1 in \cite{MDSa}; although they only treat the case where the symplectomorphism $\phi$ is the identity, we claim that everything generalizes immediately to the case of $\phi=\phi^0_1$. Introducing the energies $E(\tilde{u})$, $E(\tilde{u}^{\ell})$, $E(\tilde{u}^{\ell}_{\perp})$ to be the $L^2$-norms of the corresponding partial derivatives $\del_s\tilde{u}$, $\del_s \tilde{u}^{\ell}$ and $\del_s \tilde{u}^{\ell}_{\perp}$, we clearly have $E(\tilde{u})=E(\tilde{u}^{\ell})+E(\tilde{u}^{\ell}_{\perp})$. On the other hand, following lemma 8.1.6 in \cite{MDSa}, we know that 
\begin{eqnarray*} 
E(\tilde{u})&=&\int \tilde{u}^*\omega + \int R_{G^k}(\tilde{u})\;ds\wedge dt,\\
E(\tilde{u}^{\ell})&\geq&\int (\tilde{u}^{\ell})^*\omega + \int R_{G^k}(\tilde{u}^{\ell})\;ds\wedge dt,
\end{eqnarray*}
with $R_G$ denoting the corresponding Hamiltonian curvature form in the sense of (\cite{MDSa}, 8.1). Note that the first summands in both (in)equalities are indeed zero due to homotopical reasons and in the second case we indeed just expect an inequality, as in general $\tilde{u}^{\ell}$ itself does not satisfy the Floer equation. On the other hand, it is easy to see from the definition of the Hamiltonian curvature that $$R_{G^k}(\tilde{u}^{\ell})=R_{G^{\ell}}(\tilde{u}^{\ell})=R_{G^{\ell}}(\tilde{u}).$$ Since $R_{G^k}-R_{G^{\ell}}=R_{G^k-G^{\ell}},$ we summarizing obtain 
\begin{eqnarray*} 
E(\tilde{u}^{\ell}_{\perp}) &=& E(\tilde{u})\;-\; E(\tilde{u}^{\ell})\\
&\leq& \int R_{G^k}(\tilde{u})\;ds\wedge dt\,-\, \int R_{G^k}(\tilde{u}^{\ell})\;ds\wedge dt\\
&=& \int R_{G^k-G^{\ell}}(\tilde{u})\;ds\wedge dt. 
\end{eqnarray*}
Following remark 8.1.7 and the proof of 9.1.1 in \cite{MDSa}, we know that the last expression can be bounded by the Hofer norm of the Hamiltonian curvature of $G^k-G^{\ell}$, which itself agrees with $2\;||| G^k - G^{\ell} |||.$ \end{proof}
 
In the case of finite-dimensional nonlinearities with $\supp(\hat{\psi})\subset\{-\ell,\ldots,+\ell\}$, note that, instead of using Liouville's theorem as in section 4, we can alternatively employ the above lemma to prove that $E(\tilde{u}^{\ell}_{\perp})=0$ which in turn again immediately implies $\tilde{u}^{\ell}_{\perp}=0$. 

\begin{lemma}\label{sup-L^2} The supremum norm of $\del_s \tilde{u}^{\ell}_{\perp}$ can be bounded in terms of its $L^2$-norm and the $C^1$-norm of the differential $T\tilde{u}$ of the Floer strip $\tilde{u}$, 
$$\|\del_s \tilde{u}^{\ell}_{\perp}\|_{\infty}^3\,\leq\,\frac{4}{\sqrt{\pi}}\cdot \|\del_s \tilde{u}^{\ell}_{\perp}\|_2\cdot \|T\tilde{u}\|_{C^1}^2.$$ \end{lemma}

\begin{proof} The proof of this lemma follows from elementary estimates for the integral of the non-negative function $f=|\del_s \tilde{u}^{\ell}_{\perp}|^2:\IR\times [0,1]\to\IR$. Note that over the disk of radius $\|f\|_{\infty}/(2\|Tf\|_{\infty})$ around any maximum, the integral of $f$ can be bounded from below by $\|f\|_{\infty}/2\cdot R^2\pi$, which yields that $$\|f\|_{\infty}^3\,\leq\,\frac{8}{\pi}\cdot\|f\|_1\cdot\|Tf\|_{\infty}^2.$$ Computing all norms of $f$ in terms of the corresponding norms of $\del_s \tilde{u}^{\ell}_{\perp}$ gives the above estimate. \end{proof}

Finally, we use the admissibility of the smoothing kernel to deduce the following nondegeneracy result.

\begin{lemma}\label{nondegenerate} For $w=w^{\ell}:=\tilde{u}^{\ell}_{\perp}(s,t)$ we have that $$|\phi^0_1(w)-w|\geq \delta \cdot |w|,$$ where $\delta>0$ denotes the admissibility threshold. \end{lemma}      

\begin{proof} Since the underlying smoothing kernel is assumed to be admissible, it follows that there exists some $\delta>0$ such that $|\exp(im^2)-1|<\delta$ implies that $\hat{\psi}(m)=0$. On the other hand, by remark \ref{Liouville-2}, it follows for all $m>\ell$ that $\hat{\psi}(m)=0$ implies that $\pi_m\circ \tilde{u}^{\ell}_{\perp}=0$, where $\pi_m:\IC^{2k-2\ell}\to\IC$ denotes the projection onto the $m$.th factor. In other words, after setting $w=w^{\ell}:=\tilde{u}^{\ell}_{\perp}(s,t)$ with $w=(w_{\ell+1},\ldots,w_n)$, it follows that $w_m=0$ in the case that $|\exp(im^2)-1|<\delta$. Together with the proof of proposition \ref{free-NLS} it then follows that $$|\phi^0_1(w)-w|^2\,=\,\sum_{m=-k}^{-\ell-1} |\exp(im^2)-1|^2 |w_m|^2 + \sum_{m=\ell+1}^k |\exp(im^2)-1|^2 |w_m|^2$$ is greater or equal than $\delta^2 \cdot |w|^2$. \end{proof}

\begin{proof}\emph{(of the proposition)} For the proof we essentially use that, \emph{very informally speaking}, $G^{\IF}_t:\IP(\IF)\to{^*\IR}$ agrees up to an infinitesimal error with $G_t=G^{\IH}_t$. Together with the minimal surface property of holomorphic curves used for the finite-dimensional case, this again implies that the Floer strip $\tilde{u}:{^*\IR}\times {^*[0,1]}\to{^*\CP}^{2N}=\IP(\IF)$ has to be infinitesimally close to $\IP(\IH)$ and hence near-standard. \\ 

In order to make this idea precise we use the characterization of near-standardness in proposition \ref{characterization}. That is, we need to prove that every point in the image of $\tilde{u}=\tilde{u}^{\IF}_{n,T}:{^*\IR}\times {^*[0,1]}\to\IP(\IF)$ is infinitesimally close to ${^*\CP}^{2L}\subset{^*\CP}^{2N}=\IP(\IF)$ \emph{for every unlimited} $L\leq N\in{^*\IN}\backslash\IN$. Note that we do not need to prove that the points on the Floer strip are limited, as every point in $\IP(\IF)=\IS(\IF)/{^*U}(1)$ is limited, see remark \ref{*-spheres}. We emphasize that for our proof we employ the result of proposition \ref{no-bubbling}. \\

Let us fix $n\in\IN$ and $T\in{^*\IR^+}$ with corresponding Floer map $\tilde{u}=\tilde{u}^{\IF}_{n,T}:{^*\IR}\times {^*[0,1]}\to{^*\CP}^{2N}=\IP(\IF)$. As in the case of finite-dimensional nonlinearities we assume without loss of generality that for every unlimited $L\leq N\in{^*\IN}\backslash\IN$ the image of $\tilde{u}$ sits in a tubular neighborhood of ${^*\CP}^{2L}\subset{^*\CP}^{2N}=\IP(\IF)$; we comment on the general case at the end of this proof. In particular, as in the standard finite-dimensional case, by transfer we may write $\tilde{u}$ as a pair of maps $$(\tilde{u}^L_{\perp},\tilde{u}^L): {^*\IR}\times {^*[0,1]}\to {^*\IC}^{2N-2L}\times {^*\CP}^{2L}$$  for every unlimited $L\leq N\in{^*\IN}\backslash\IN$. Using proposition \ref{characterization} it suffices to prove that for \emph{every} unlimited $L\in{^*\IN}\backslash\IN$ with $L\leq N$ we have $\tilde{u}^L_{\perp}\approx 0$, that is, $\tilde{u}^L_{\perp}(s,t)\approx 0$ for all $(s,t)\in{^*\IR}\times {^*[0,1]}$. \\

In order to prove this, observe first that by proposition \ref{approximation} we know for the non-standard extension of the Hofer norm that $$|||G^N-G^L|||\leq |||{^*G}-G^N|||+|||{^*G}-G^L||| \approx 0$$ as $L$ and $N$ are unlimited. After applying the transfer principle to the first lemma \ref{L^2-Hofer}, we get that the (non-standard extension of the) $L^2$-norm ${^*\|\del_s \tilde{u}^L_{\perp}\|_2}$ of $\del_s\tilde{u}^L_{\perp}$ must be infinitesimal as well. On the other hand, after applying the transfer principle to the second lemma \ref{sup-L^2}, we get for the (non-standard extension of the) supremum norm of $\del_s\tilde{u}^L_{\perp}$ that $${^*\|\del_s \tilde{u}^L_{\perp}\|_{\infty}}^3\,\leq\,\frac{4}{\sqrt{\pi}}\cdot {^*\|\del_s \tilde{u}^L_{\perp}\|_2}\cdot {^*\|T\tilde{u}\|_{C^1}}^2.$$ Since we already know that the (nonstandard) $C^1$-norm of $T\tilde{u}$ is indeed limited by proposition \ref{no-bubbling}, and as the product of an infinitesimal number with a limited number is still infinitesimal by proposition \ref{limited-infinitesimal}, we can actually deduce that also the (non-standard extension of the) supremum norm of $\del_s\tilde{u}^L_{\perp}$ must be infinitesimal as well, that is, $\del_s\tilde{u}^L_{\perp}(s,t)\approx 0$ for all points $(s,t)\in{^*\IR}\times {^*[0,1]}$.\\

It remains to be shown that $\del_s\tilde{u}^L_{\perp}\approx 0$ indeed implies $\tilde{u}^L_{\perp}\approx 0$. To this end observe first that, by $\CR^T_G\tilde{u}=0$, we know that $\del_s\tilde{u}^L_{\perp} \approx 0$ implies that $\del_t \tilde{u}^L_{\perp} \approx (X^{G,\IF}_t(\tilde{u}))^L_{\perp}$, where $(X^{G,\IF}_t(\tilde{u}))^L_{\perp}$ denotes the projection of the symplectic gradient $X^{G,\IF}_t(\tilde{u})$ of $G^{\IF}_t$ to ${^*\IC}^{2N-2L}$. Since we again have $(X^{G,\IF}_t(\tilde{u}))^L_{\perp}=(X^{G,N}_t(\tilde{u}))^L_{\perp}\approx (X^{G,L}_t(\tilde{u}))^L_{\perp}=0$ by proposition \ref{approximation} and proposition \ref{convergence}, it follows that $\del_t \tilde{u}^L_{\perp}\approx 0$. But this immediately implies that  $$\tilde{u}^L_{\perp}(s,t)\approx \tilde{u}^L_{\perp}(s,t+1)=\phi^0_1(\tilde{u}^L_{\perp}(s,t))$$ for all $(s,t)\in{^*\IR}\times{^*S^1}$. \\

In order to finish the proof, we need to apply the transfer principle to the third lemma \ref{nondegenerate} in order to get that $$|\tilde{u}^L_{\perp}(s,t)|\,\leq\, \delta^{-1}\cdot |\phi^0_1(\tilde{u}^L_{\perp}(s,t))-\tilde{u}^L_{\perp}(s,t)|.$$ Since $\delta^{-1}>0$ is standard and hence limited and $|\phi^0_1(\tilde{u}^L_{\perp}(s,t))-\tilde{u}^L_{\perp}(s,t)|$ is already known to be infinitesimal, we can again employ proposition \ref{limited-infinitesimal} in order to finally deduce that $\tilde{u}^L_{\perp}(s,t)\approx 0$ for all $(s,t)\in{^*\IR}\times{^*[0,1]}$. \\

Note that in the case when the image of $\tilde{u}$ is not yet known to be fully contained in a tubular neighborhood of ${^*\CP}^{2L}$, our proof can be used to show that the part contained in it must be infinitesimally close to ${^*\CP}^{2L}$. Together with the asymptotic condition and the Lipschitz continuity of $\tilde{u}$ (with limited Lipschitz constant given by its non-standard $C^1$-norm) this proves that the full image of $\tilde{u}$ has to be infinitesimally close. \end{proof}

\begin{remark} We want to compare this with Grossman's infinite-dimensional counter-example to the Hopf-Rinow theorem. Note that the latter states that every complete Riemannian manifold carries a minimizing geodesic. For his counter-example Grossman defines the sequence $(a_n)_{n\in\IN}$ by $a_0=1$ and $a_n=1+2^n$ for $n\geq 1$, considers the infinite-dimensional ellipsoid $$Q=\Big\{(x_n)_{n\in\IN}: \sum_{n=0}^{\infty} \frac{x_n^2}{a_n^2}=1\Big\}\subset\ell^2(\IR)$$ and shows that there does not exist a minimizing geodesic between the points $+e_0$ and $-e_0$ on $Q$, where $(e_n)_{n\in\IN}$ denotes the canonical complete orthonormal basis of $\ell^2(\IR)$. On the other hand, for every finite-dimensional subellipsoid $$Q^k=\Big\{(x_n)_{n=0}^k: \sum_{n=0}^k \frac{x_n^2}{a_n^2}=1\Big\}=Q\cap\IR^k$$ it is not hard to see that the corresponding minimizing geodesic $\gamma^k$ between $+e_0$ and $-e_0$ lies on the intersection of $Q^k$ with the $e_0$-$e_k$-plane. Choosing an unlimited *-natural number $N\in{^*\IN}\backslash\IN$ it follows as in proposition \ref{existence} that $Q$ is contained in $Q^N$ up to an infinitesimal error. While there exists a minimizing geodesic $\gamma^N$ on $Q^N$ between $+e_0$ and $-e_0$ by transfer, in contrast to our result this does \emph{not} provide us with a minimizing geodesic on $Q$: Since $\gamma^N$ still lies on the intersection of $Q^N$ with the $e_0$-$e_N$-plane, it follows that almost all points on $\gamma^N$ have a non-infinitesimal distance to $Q^L=Q^N\cap{^*\IR}^L$ for all $L<N$. But by proposition \ref{characterization} this means that $\gamma^N$ is \emph{not} near-standard. \end{remark}  

Since we have established that, for every $n\in\IN$, the map $\tilde{u}=\tilde{u}^{\IF}_n:{^*\IR}\times {^*[0,1]}\to\IP(\IF)$ is near-standard, we can apply the standard part map from definition \ref{near-standard} to obtain a standard map $\tilde{u}^{\IH}_n:={^{\circ}\tilde{u}}:\IR\times [0,1]\to\IP(\IH)$. Although the map $\tilde{u}^{\IF}_n$ is only differentiable in the non-standard sense of remark \ref{*-differentiable}, we will show below that the limitedness of the derivatives of $\tilde{u}$ is indeed sufficient to show that ${^{\circ}\tilde{u}}$ is smooth in the standard sense. In the last proposition we prove that, for unlimited $T$,  $\tilde{u}^{\IH}_n={^{\circ}\tilde{u}}$ is a Floer strip as claimed in the main theorem; in particular, there exists sequences $(s^{\pm}_{\alpha})$ of positive real numbers such that $\tilde{u}^{\IH}_n(s^{\pm}_{\alpha},\cdot)$ converges to fixed points $u^0_n$, $u^1_n$ of the time-one flows $\phi^0_1$, $\phi_1$, respectively.

\begin{lemma}\label{Floer-smooth} The map $\tilde{u}^{\IH}_n:={^{\circ}\tilde{u}}:\IR\times [0,1]\to\IP(\IH)$ is smooth in the standard sense, where the standard derivatives of $\tilde{u}^{\IH}_n$ and the non-standard derivatives of $\tilde{u}^{\IF}_n$ are related via $$\del_s\tilde{u}^{\IH}_n(s,t)={^{\circ}(\del_s\tilde{u}^{\IF}_n(s,t))},\,\,\del_t\tilde{u}^{\IH}_n(s,t)={^{\circ}(\del_t\tilde{u}^{\IF}_n(s,t))}$$ for all $(s,t)\in\IR\times [0,1]$.\end{lemma}

\begin{proof} For proving the differentiability of ${^{\circ}\tilde{u}}$ at some $(s,t)\in\IR\times [0,1]$, note that it suffices to prove that the map ${^{\circ}\tilde{u}_m}: B^2_r(0)\to\IH/\IC$, obtained by applying the standard part map to the map $\tilde{u}_m=\varphi_m^{-1}\circ\tilde{u}:{^*B}^2_r(0)\to\IF/{^*\IC}={^*\IC}^{2N}$ is differentiable at $z=0$, where $m\in\{0,\ldots,2N\}$ is chosen such that $\tilde{u}(s,t)\in\varphi_m({^*\IC}^{2N})$; recall from above that the radius $r>0$ can indeed chosen to be a standard positive real number. Fix a direction $\theta\in S^1$ and define the map $f=f_{\theta}:{^*[0,r)}\to\IF/{^*\IC} ={^*\IC}^{2N}$ by $f(x)=\tilde{u}_m(x\cdot e^{i\theta})$. Since $\tilde{u}_m$ is near-standard in the sense of proposition \ref{near-standard} and *-smooth with limited ${^*\|\tilde{u}_m\|_{C^{\ell}}}$ for all $\ell\in\IN$, the same holds true for $f$, and it suffices to prove that ${^{\circ}f}:[0,r)\to\IH/\IC$ is differentiable at $0$. \\

As the first step we prove that ${^{\circ}f}$ is Lipschitz continuous. After applying the transfer principle to the intermediate value theorem and using that $f$ is differentiable in the non-standard sense, note that we have for every $x<y\in{^*[0,r)}$ that $$\frac{f(y)-f(x)}{y-x}=f'(w)\,\,\textrm{for some}\,\,w\in{^*[x,y]},$$ which implies that $|f(y)-f(x)|\leq{^*\|f\|_{C^1}}\cdot |y-x|$. By applying the standard part map, it follows that ${^{\circ}f}$ is Lipschitz continuous, $$|{^{\circ}f}(y)-{^{\circ}f}(x)|\leq c_1\cdot |y-x|,$$ where the positive real number $c_1\geq 0$ is the standard part of the limited number ${^*\|f\|_{C^1}}\in{^*\IR}^+\cup\{0\}$. After showing that ${^{\circ}f}$ is continuous in the standard sense, we now prove that it is differentiable. For this we observe that, by the same arguments as used for $f$, its first derivative $f':[0,r)\to\IF/\IC$ is Lipschitz with limited Lipschitz constant given by ${^*\|f\|_{C^2}}$. For the difference quotient used to establish the Lipschitz continuity, this can be used to prove that $$\Big|\frac{f(y)-f(x)}{y-x}-f'(x)\Big|=|f'(w)-f'(x)|\leq {^*\|f\|_{C^2}}\cdot |y-x|$$ using that $w\in{^*[x,y]}$. It follows that for every standard $\eps>0$ there exists a standard $\delta=\eps/c_2>0$ with $c_2=\max\{{^{\circ}({^*\|f\|_{C^2}})},1\}$ with the property that $$|y-x|<\delta\,\,\textrm{implies}\,\,\Big|\frac{f(y)-f(x)}{y-x}-f'(x)\Big|<\eps.$$ Now using that the near-standardness of $f$ implies that the above difference quotient is near-standard with standard part given by $${^{\circ}\Big(\frac{f(y)-f(x)}{y-x}\Big)}=\frac{{^{\circ}f}(y)-{^{\circ}f}(x)}{y-x},$$ it follows that $$|y-x|<\delta\,\,\textrm{implies}\,\,\Big|\frac{{^{\circ}f}(y)-{^{\circ}f}(x)}{y-x}-f'(x)\Big|<\eps.$$ But this proves that ${^{\circ}f}$ is differentiable at $x\in[0,r)$ in the standard sense with derivative given by the standard part of $f'(x)$; in particular, we see a posteriori that $f'$ has to be near-standard itself. \\  

On the other hand, after replacing $f$ by $f'$ and employing the limitedness of ${^*\|f'\|_{C^2}}\leq{^*\|f\|_{C^3}}$, one can successively prove that ${^{\circ}f}$ is infinitely often differentiable, that is, smooth in the standard sense.  Finally, observe that the latter also proves that the standard derivatives of $\tilde{u}^{\IH}_n$ and the non-standard derivatives of $\tilde{u}^{\IF}_n$ are related via $\del_s\tilde{u}^{\IH}_n={^{\circ}(\del_s\tilde{u}^{\IF}_n)}$, $\del_t\tilde{u}^{\IH}_n={^{\circ}(\del_t\tilde{u}^{\IF}_n)}$. \end{proof}

With this we can now finish the proof of the main theorem with the following two propositions.

\begin{proposition} The map $\tilde{u}^{\IH}_n:\IR\times [0,1]\to\IP(\IH)$ is a Floer strip as in the main theorem, that is, it satisfies the perturbed Cauchy-Riemann equation $\CR\tilde{u}^{\IH}_n=\varphi(s)\cdot\nabla G^{\IH}_t(\tilde{u}^{\IH})$, the periodicity condition $\tilde{u}^{\IH}_n(\cdot,1)=\phi^0_1(\tilde{u}^{\IH}_n(\cdot,0))$, and it connects the reference fixed point $u^0_n$ of the free Schr\"odinger equation with a fixed point $u^1_n$ of the given nonlinear Schr\"odinger equation of convolution type in the sense that there exist two sequences $(s_k^{\pm})_{k\in\IN}$ of real numbers, $s_k^{\pm}\to\pm\infty$ with $\tilde{u}^{\IH}(s_k^-,\cdot)\to u^0_n$ and $\tilde{u}^{\IH}(s_k^+,\cdot)\to u^1_n$ as $k\to\infty$. \end{proposition}

\begin{proof} Starting with the periodicity condition, note that by proposition \ref{*-free NLS} we have for all $s\in\IR$ that $$\phi^0_1(\tilde{u}^{\IH}(s,0))={^{\circ}({^*\phi^0_1}(\tilde{u}^{\IF}(s,0)))}={^{\circ}\tilde{u}^{\IF}(s,1)}=\tilde{u}^{\IH}(s,1).$$ In the same spirit, note that by combining proposition \ref{Floer-smooth} with proposition \ref{approximation} we have $$\CR\tilde{u}^{\IH}={^{\circ}\big(\CR\tilde{u}^{\IF}\big)}={^{\circ}\big(\varphi_T(s)\nabla G^{\IF}(\tilde{u}^{\IF})\big)}=\varphi(s)\nabla G^{\IH}(\tilde{u}^{\IH}),$$ when we define $\varphi:\IR\to [0,1]$ by $\varphi(s)={^{\circ}\varphi_T(s)}$; note that $\varphi_T(s)=1$ for all $s\in [0,2T]$ implies that $\varphi(s)=1$ for all $s\in\IR^+_0$ when $T\in{^*\IR}^+$ is unlimited. \\

For the asymptotic condition consider first a Floer strip $\tilde{u}=\tilde{u}_n=\tilde{u}^k_{n,T}:\IR\times [0,1]\to\CP^{2k}$ as in proposition \ref{strips}. By the finiteness of energy it follows, as in the proof of \ref{finite-dim}, that for every $\alpha\in\IN$ with $\alpha\leq T$ there exist $-2\alpha\leq s_{\alpha}^-\leq -\alpha$, $\alpha\leq s_{\alpha}^+\leq 2\alpha$ with $$\int_0^1 |\del_t\tilde{u}(s_{\alpha}^{\pm},t) - \varphi_T(s_{\alpha}^{\pm}) X^{G,k}_t(\tilde{u}(s_{\alpha}^{\pm},t))|^2\;dt\,<\,\frac{\pi}{2\alpha},$$ and hence using the Cauchy-Schwarz inequality $$d(\tilde{u}(s_{\alpha}^-,0),\phi^0_1(\tilde{u}(s_{\alpha}^-,0))),\,\,d(\tilde{u}(s_{\alpha}^+,0),\phi_1(\tilde{u}(s_{\alpha}^+,0)))\,<\,\sqrt{\frac{\pi}{2\alpha}},$$ where $d$ denotes the Riemannian distance function on $\CP^{2k}\subset\IP(\IH)$. By letting $k\in\IN$ vary and applying the transfer principle to this statement  and fixing the unlimited $T\in{^*\IR}^+$ from before, it follows that there exists sequences $(s_{\alpha}^{\pm})_{\alpha\in\IN}$ of real numbers with $-2\alpha\leq s_{\alpha}^-\leq -\alpha$, $\alpha\leq s_{\alpha}^+\leq 2\alpha$ with $${^*d}(\tilde{u}^{\IF}(s_{\alpha}^-,0),{^*\phi^0_1}(\tilde{u}^{\IF}(s_{\alpha}^-,0))),\,\,{^*d}(\tilde{u}^{\IF}(s_{\alpha}^+,0),{^*\phi_1}(\tilde{u}^{\IF}(s_{\alpha}^+,0)))\,<\,\sqrt{\frac{\pi}{2\alpha}},$$ where ${^*d}$ is the *-extension of the distance $d$ on $\IP(\IH)$. After applying the standard part map, this in turn implies that $$d(\tilde{u}^{\IH}(s_{\alpha}^-,0),\phi^0_1(\tilde{u}^{\IH}(s_{\alpha}^-,0))),\,\,d(\tilde{u}^{\IH}(s_{\alpha}^+,0),\phi_1(\tilde{u}^{\IH}(s_{\alpha}^+,0)))\,<\,\sqrt{\frac{\pi}{2\alpha}}.$$ By applying the classical diagonal sequence argument and possibly passing to a subsequence, we can assume that for all $\ell\in\IN$ the sequences of projected points $(\pi_{\ell}\circ\tilde{u}^{\IH})(s_{\alpha}^{\pm},0)\in\CP^{2\ell}$ are convergent, where $\pi_{\ell}:\IP(\IH)\to\CP^{2\ell}$ denote the canonical projections onto the finite-dimensional projective spaces. \\

In order to prove that the sequences $(s_{\alpha}^{\pm})_{\alpha\in\IN}$ have the desired property, we now make use again of proposition \ref{step1}. Fix some standard $\eps>0$. Since $\sup\{d(\tilde{u}^{\IF}(s,t),{^*\CP}^{2L}):\;(s,t)\in{^*\IR}\times {^*[0,1]}\}\approx 0$ for every unlimited $L\in\IN$ and hence $\sup\{{^*d}(\tilde{u}^{\IF}(s,t),{^*\CP}^{2\ell}):\;(s,t)\in{^*\IR}\times{^*[0,1]}\} <\eps/3$ for all $\ell\geq L$, it follows from the spillover principle in proposition \ref{spillover} that there must exist a standard $L\in\IN$ with $\sup\{{^*d}(\tilde{u}^{\IF}(s,t),{^*\CP}^{2\ell}):\,(s,t)\in{^*\IR}\times{^*[0,1]}\}<\eps/3$ and hence $\sup\{d(\tilde{u}^{\IH}(s,t),\CP^{2\ell}):\,(s,t)\in\IR\times [0,1]\}<\eps/3$ for all $\ell\geq L$. On the other hand, fixing some $\ell\geq L$ we find some $n_0 \in\IN$ such that for all $\alpha,\beta\geq n_0$ we have $d((\pi_{\ell}\circ\tilde{u}^{\IH})(s_{\alpha}^{\pm},0), (\pi_{\ell}\circ\tilde{u}^{\IH})(s_{\beta}^{\pm},0))<\eps/3.$ Together with $$d(\tilde{u}^{\IH}(s_{\alpha}^{\pm},0),(\pi_{\ell}\circ\tilde{u}^{\IH})(s_{\alpha}^{\pm},0)),\,\,d(\tilde{u}^{\IH}(s_{\beta}^{\pm},0),(\pi_{\ell}\circ\tilde{u}^{\IH})(s_{\beta}^{\pm},0))\,<\,\eps/3,$$ it follows summarizing that $$d(\tilde{u}^{\IH}(s_{\alpha}^{\pm},0),\tilde{u}^{\IH}(s_{\beta}^{\pm},0)))\,<\,\eps/3+\eps/3+\eps/3=\eps,$$ proving that the sequence of points $(\tilde{u}^{\IH}(s_{\alpha}^{\pm},0))_{\alpha\in\IN}$ has the Cauchy property and hence converges by the completeness of $\IP(\IH)$. Together with $$d(\tilde{u}^{\IH}(s_{\alpha}^-,0),\phi^0_1(\tilde{u}^{\IH}(s_{\alpha}^-,0)))\,<\,\sqrt{\frac{\pi}{2\alpha}}\,\to\, 0$$ and $$d(\tilde{u}^{\IH}(s_{\alpha}^+,0),\phi_1(\tilde{u}^{\IH}(s_{\alpha}^+,0)))\,<\,\sqrt{\frac{\pi}{2\alpha}}\,\to\, 0$$ as $\alpha\to\infty$, it follows that the limit points are fixed points of $\phi^0_1$ and $\phi_1$, respectively. \end{proof}

In order to complete the proof of the main theorem it just remains to show the following 

\begin{proposition} For $m>n\geq 4$ the corresponding fixed points $u^1_m$ and $u^1_n$ are different. \end{proposition}

\begin{proof} We will prove this by applying the transfer principle to the corresponding proposition \ref{action} for finite-dimensional nonlinearities. For this we consider the *-extension $({^*s^+_{\alpha}})_{\alpha\in{^*\IN}}$ of the sequence $(s^+_{\alpha})_{\alpha\in\IN}$. Since ${^{\circ}\tilde{u}^{\IF}_m(s^+_{\alpha},0)}\to u^1_m$ as $\alpha\to\infty$, it follows that for every standard $\eps>0$ there exists some $\alpha_0\in\IN$ such that ${^*d}(\tilde{u}^{\IF}_m(s^+_{\alpha},0),u^1_m)<\eps$ for all $\alpha\geq\alpha_0$. Since every unlimited *-natural number $A\in{^*\IN}\backslash\IN$ is greater than any standard $\alpha_0$, it follows from the transfer principle that $\tilde{u}^{\IF}_m({^*s^+_A},0)\approx u^1_m$ for all unlimited *-natural number $A\in{^*\IN}\backslash\IN$. On the other hand, since for all natural numbers $\alpha$ the *-extension of the $L^2$-norm of $\del_t\tilde{u}^{\IF}_m(s^+_{\alpha}) +  X^{G,\IF}_t(\tilde{u}^{\IF}_m(s^+_{\alpha}))$ is less than $\sqrt{\pi/2\alpha}$, it follows from the spillover principle in proposition \ref{spillover} that there must an unlimited *-natural number $A\in{^*\IN}\backslash\IN$ such that the *-extension of the $L^2$-norm of $\del_t\tilde{u}^{\IF}_m({^*s^+_A}) +  X^{G,\IF}_t(\tilde{u}^{\IF}_m({^*s^+_A}))$ is less than $\sqrt{\pi/2A}$, that is, infinitesimal. Let us define $s:={^*s^+_A}$ and note that in the same way we find $s'\in{^*\IR^+}$ such that $\tilde{u}^{\IF}_n(s',0)\approx u^1_n$ and the *-extension of the $L^2$-norm is again infinitesimal. By applying the transfer principle to proposition \ref{action}, it now follows for $m>n\geq 4$ that $|{^*\mathcal{A}}(\tilde{u}^{\IF}_m(s))-{^*\mathcal{A}}(\tilde{u}^{\IF}_n(s'))|>1$, employing the *-extension ${^*\mathcal{A}}$ of the symplectic action defined in proposition \ref{action}. The proof is complete after we have shown the following\\

\noindent\emph{Claim: If $\tilde{u}^{\IF}_m(s)\approx\tilde{u}^{\IF}_n(s')$, then ${^*\mathcal{A}}(\tilde{u}^{\IF}_m(s))\approx{^*\mathcal{A}}(\tilde{u}^{\IF}_n(s'))$.}\\

So let us assume to the contrary that $\tilde{u}^{\IF}_m(s,t)\approx\tilde{u}^{\IF}_n(s',t)$ for all $t\in{^*[0,1]}$. Then it first follows as in the proof of proposition \ref{approximation} from the continuity of the gradient that $$X^{G,\IF}_t(\tilde{u}^{\IF}_m(s,t))=i\nabla G^{\IF}_t(\tilde{u}^{\IF}_m(s,t))\approx i\nabla G^{\IF}_t(\tilde{u}^{\IF}_n(s',t))= X^{G,\IF}_t(\tilde{u}^{\IF}_n(s',t)).$$ Together with the fact that the *-extension of the $L^2$-norm of $\del_t \tilde{u}^{\IF}_m(s)-X^{G,\IF}_t(\tilde{u}^{\IF}_m(s))$ and $\del_t \tilde{u}^{\IF}_n(s')-X^{G,\IF}_t(\tilde{u}^{\IF}_n(s'))$ is infinitesimal, it follows that also $${^*\|\del_t \tilde{u}^{\IF}_m(s)-\del_t \tilde{u}^{\IF}_n(s')\|_2}\approx 0.$$ In order to finish the proof of the claim, we apply the transfer principle to the following continuity result for the symplectic action in finite dimensions. To this end, assume that $\tilde{u}_m,\tilde{u}_n:\IR\times [0,1]\to\CP^{2k}$ are Floer strips in the finite-dimensional complex projective space for some $k\in\IN$ for the given Hamiltonian $G=G^k$ as in proposition \ref{strips}. Then it is an easy exercise to show that there exists a constant $c>0$ such that $|\mathcal{A}(\tilde{u}_m(s))-\mathcal{A}(\tilde{u}_n(s'))|$ is bounded by $$c\cdot \big(\sup\{d(\tilde{u}_m(s,t),\tilde{u}_n(s',t)):t\in [0,1]\}+\|\del_t \tilde{u}_m(s)-\del_t \tilde{u}_n(s')\|_2\big),$$ where the constant just depends on the $C^1$-norm of $G$ and is independent of the dimension $k$ of the target manifold. By transfer, it follows that the same inequality continues to hold for the Floer strips $\tilde{u}^{\IF}_m,\tilde{u}^{\IF}_n:{^*\IR}\times {^*[0,1]}\to\IP(\IF)={^*\CP^{2N}}$ when the symplectic action, the Riemann distance and the $L^2$-norm are replaced by their *-extensions. Since $\tilde{u}^{\IF}_m(s)\approx\tilde{u}^{\IF}_n(s')$ and ${^*\|\del_t \tilde{u}^{\IF}_m(s)-\del_t \tilde{u}^{\IF}_n(s')\|_2}\approx 0$, it hence follows that ${^*\mathcal{A}}(\tilde{u}^{\IF}_m(s))\approx{^*\mathcal{A}}(\tilde{u}^{\IF}_n(s'))$ as desired. \end{proof}

\section{Appendix: Non-standard model theory}

In this section we provide an outline of all the background and relevant definitions and statements about nonstandard analysis that the reader needs to know in order to follow the rest of the paper. Here we describe the original model-theoretic approach of Robinson (\cite{R}), outlined in the excellent expositions \cite{L1}, \cite{L2} as well as in \cite{Ke}, to which we refer and which shall also be consulted for more details and background. \\

Believing in the axiom of choice it is well-known, see e.g. (\cite{L2}, theorem 2.9.10), that there exist non-standard models of mathematics in which, on one side, one can do the same mathematics as before (transfer principle) but, on the other side, all sets behave like compact sets (saturation principle). The idea is to successively introduce new ideal objects such as infinitely small and large numbers. The proof of existence of the resulting polysatured model is then performed in complete analogy to the proof of the statement that every field has an algebraic closure, by employing the axiom of choice. \\

A model of mathematics $V$ is a family of sets which is rich enough in order to do all the mathematics that one has in mind. Since for existence proof of non-standard models it is crucial that $V$ is still a set in the sense of set theory, there are (abstract) sets which are not in $V$. Below we show how to define such a set $V$ which contains all mathematical entities that we need for our proof. For most of the upcoming definitions and theorems on the general background on model theory we refer the reader to \cite{L2} as well as \cite{Ke}. The first definition is taken from the appendix in (\cite{L2}, section 2.9). 

\begin{definition} A sequence $V=(V_n)_{n\in\IN}$ of ordered sets $V_n$, $n\in\IN$ is called a \emph{model} if the elements in $V_n$ are sets formed from the elements in $V_0,\ldots,V_{n-1}$, i.e.,  $V_n\subset\PP(V_0\cup\ldots\cup V_{n-1})$ and $V_0$, called the set of urelements, does not contain elements from higher sets, i.e., $V_0\cap\bigcup_{n\geq 1} V_n=\emptyset$. \end{definition}

By choosing the model $V=(V_n)_{n\in\IN}$ large enough, one can ensure that the model contains all mathematical entities that one wants to work with. Apart from assuming that every subset formed from elements in $V_0,\ldots,V_{n-1}$ is in $V_n$, below we show explicitly that for our proof it turns out to be sufficient to take the real numbers as urelements, i.e., $V_0=\IR$.

\begin{definition} We call $V=(V_n)_{n\in\IN}$ the \emph{standard model} if the urelements are the real numbers, $V_0=\IR$, and the model is full in the sense that $V_n=\PP(V_0\cup\ldots\cup V_{n-1})$. \end{definition}

In what follows, let $V=(V_n)_{n\in\IN}$ denote the standard model. As discussed in (\cite{L2}, 2.9), it follows that $$V(\IR)=\bigcup_{n=0}^{\infty} V_n(\IR)\;\;\textrm{with}\;\; V_n(\IR)=\IR\cup V_n\;\;\textrm{for all}\;\;n\in\IN$$ is the superstructure over the real numbers in the sense of (\cite{L2}, definition 2.1.1) and (\cite{Ke}, definition 15.4). Note that, for $n>1$, we have for every full model that $V_{n-1}\subset V_n$. \\

\begin{proposition} The standard model $V=(V_n)_{n\in\IN}$ contains (isomorphic copies of) all mathematical entities that appear in our proof.  
\end{proposition}

\begin{proof} Since in analysis one considers sets of functions which themselves can be viewed as sets built from the real numbers, the superstructure over the real numbers contains all mathematical entities that one needs to do analysis, see (\cite{Ke}, section 15B). In particular, if $a$ and $b$ are sets in $V_n$, then every function $f:a\to b$ is an element of $V_{n+2}$ and every set of functions $f:a\to b$ is an element of $V_{n+3}$. For this it suffices to observe that, in set theory, a function $f:a\to b$ is identified with the subset $\{(x,f(x)): x\in a\}$ of $a\times b$ and for each $x\in a$, $y\in b$ the tuple $(x,y)$ is defined as the set $\{x,\{x,y\}\}$. \end{proof} 

In order to show that Floer's existence result of symplectic fixed points and pseudo-holomorphic strips indeed continues to hold in infinite dimensions, we will use that, by abstract model theory, his statement also holds in the non-standard model which we are going to discuss below. To make the underlying transfer principle precise, we quickly recall all the necessary background from first-order predicate logic that is needed. \\

The idea is that, just like all mathematical entities that we need are contained in the standard model $V=(V_n)_{n\in\IN}$, all statements that we will transfer can be formalized in first-order logic, that is, they are sentences in the language $\LL_V$ for our standard model $V$. In the same way as the details in the precise definition of models are not ultimatively important in order to understand the strategy of our proof, we continue to recall all needed foundations from logic for the sake of completeness of the exposition. For the following definitions we continue to refer to the appendix in (\cite{L2}, section 2.9) as well as (\cite{Ke}, section 15B).

\begin{definition} The \emph{alphabet} of the \emph{language} $\LL_V$ of the model $V=(V_n)_{n\in\IN}$ consists of the logical symbols $\vee$, $\neg$, $\exists$, $=$, $\in$, a countable number of variables, the elements in $V_{<\infty}:=\bigcup_{n\in\IN} V_n$ as parameters, and auxiliary symbols like parentheses. 
\end{definition}

\begin{definition} A \emph{sentence} in the \emph{language} $\LL_V$ of the model $V=(V_n)_{n\in\IN}$ is build inductively from the following rules: \begin{itemize}
\item[i)] If $a,b\in V_{<\infty}$, then $a\in b$ and $a=b$ are sentences in $\LL_V$.
\item[ii)] If $A$ and $B$ are sentences in $\LL_V$, then $A\vee B$ and $\neg A$ are sentences in $\LL_V$.
\item[iii)] Let $A$ be a sentence in $\LL_V$ and $a,b\in V_{<\infty}$ be parameters in $\LL_V$. If $x$ is a variable not occurring in $A$, then ${\exists x\in a}\,\, A_b(x)$ is a sentence in $\LL_V$, where $A_b(x)$ is obtained from $A$ by replacing each occurrence of the parameter $b$ in $A$ by the variable $x$. 
\end{itemize}
Every $A(x)=A_b(x)$ as in part iii) with a free variable $x$ is called a \emph{formula} in $\LL_V$. Furthermore, for every parameter $a\in V_{<\infty}$, by $A(x)(a)$ we denote the new sentence in $\LL_V$ obtained by replacing the variable $x$ by the parameter $a$. 
\end{definition}

Whether a sentence $A$ holds true in the model $V$, written $V\models A$, is decided using the usual interpretation for sentences in set theory, see (\cite{L2}, 2.9), (\cite{Ke}, 15B). \\

Using the axiom of choice one can prove that there exists a so-called \emph{non-standard model} in which the same mathematics hold true but in which every set from $V$ can be viewed as a compact set. More precisely, after reformulating (\cite{L2}, theorem 2.9.10), we have the following 

\begin{theorem}\label{non-standard model}
Given the standard model $V=(V_n)_{n\in\IN}$ there exists a corresponding non-standard model $W=(W_n)_{n\in\IN}$, together with an embedding $*: V_{<\infty}\rightarrow W_{<\infty}$ respecting the filtration, i.e. $*_n: V_n\rightarrow W_n$, satisfying the following two important principles. \\
\begin{itemize}
\item\emph{Transfer principle:} If a sentence $A$ holds in the language $\LL_V$ of the model $V$, $V\models A$, then the corresponding sentence ${^*}A$, obtained by replacing the parameters from $V$ by their images in $W$ under $*$, holds in the language $\LL_W$ of the model $W$, $W\models {^*}A$. \\
\item\emph{Saturation principle:} If $(a_i)_{i\in I}$ is a collection of sets in $W$, indexed by a set $I$ in $V$, and satisfying $a_{i_1}\cap\ldots\cap a_{i_n}\neq\emptyset$ for all $i_1,\cdots,i_n\in I$, $n\in\IN$ \emph{(finite intersection property)}, then also the common intersection of all $a_i$, $i\in I$ is non-empty, $\bigcap_{i\in I} a_i\neq\emptyset$. \\
\end{itemize}
\end{theorem}

\begin{proof} Since in the references the theorem is not precisely stated in the above form, let us quickly describe how it can be deduced from \cite{L2}. In (\cite{L2}, theorem 2.9.10) it is claimed that there exists a so-called monomorphism from the superstructure $V(\IR)$ over $\IR$ into the superstructure $V({^*}\IR)$ over the set ${^*}\IR$ of non-standard real numbers. The latter are defined explicitly as equivalence classes of sequences of real numbers using the axiom of choice in (\cite{L2}, definition 1.2.3). Note that by (\cite{L2}, definition 2.4.3 and remark 2.4.4) the property of the map $*: V(\IR)\to V({^*}\IR)$ being a monomorphism is equivalent to the transfer principle, in particular, the latter indeed implies that $*$ respects the filtration. On the other hand, the fact that the formulation of the saturation principle given here is equivalent to the definition in (\cite{L2}, definition 2.9.1) is proven in (\cite{L2}, theorem 2.9.4), noticing that, by the definition of the cardinal number $\kappa^+$ appearing in (\cite{L2}, theorem 2.9.10), every set in $V$ is $\kappa^+$-small. Since the saturation property is only assumed when all sets $a_i$, $i\in V$ are \emph{internal} in the sense of (\cite{L2}, definition 2.8.1), that is, when they are elements in the $*$-image ${^*}V_n(\IR)\subset V_n({^*}\IR)$ of the set $V_n(\IR)\in V_{n+1}(\IR)$, we follow the strategy in the appendix of (\cite{L2}, section 2.9) and define the non-standard model $W=(W_n)_{n\in\IN}$ by setting $W_n:={^*}V_n={^*}V_n(\IR)$ for all $n\in\IN$. In particular, every set in the non-standard model $W=(W_n)_{n\in\IN}$ is internal.      
\end{proof}

In what follows we follow the usual conventions and write ${^*}a:=*(a)$ for every set $a\in V_{<\infty}\backslash V_0$ and identify $a:=*(a)$ for every urelement $a\in V_0=\IR$. \\
 
\begin{definition} A set $a$ is called
\begin{itemize}
\item[i)] \emph{internal} if $a\in W_{<\infty}$,
\item[ii)] \emph{standard} if $a={^*}b:=*(b)\in W_{<\infty}$ for some $b\in V_{<\infty}$.
\item[iii)]\emph{external} if $a$ is not internal.
\end{itemize} 
\end{definition}

We start with some immediate consequences of the transfer principle, see (\cite{L2}, proposition 2.4.6).

\begin{proposition} Let $a,b$ be sets in $V_{<\infty}$. Then we have
\begin{itemize}
\item[i)] $a=b$ if and only if ${^*}a={^*}b$,
\item[ii)] $a\in b$ if and only if ${^*}a\in {^*}b$,
\item[iii)] $a\subset b$ if and only if ${^*}a\subset {^*}b$,
\item[iv)] $f:a\to b$ if and only if ${^*}f: {^*}a\to {^*}b$.
\end{itemize}
\end{proposition}

These in turn lead to the following 
\begin{corollary}\label{extension} It follows  
\begin{itemize}
\item[i)] $*: V_{<\infty}\to W_{<\infty}$ is an embedding.
\item[ii)] For every set $b\in V_{<\infty}$ we have that  ${^*}[b]:=\{{^*}a: a\in b\}\subset {^*}b$. 
\item[iii)] For every function $f:a\to b$ we have that ${^*}f: {^*}a\to {^*}b$ is an extension of $f$ in the sense that for all $c\in a$ we have ${^*}(f(c))=({^*}f)({^*}c)\in {^*}b$.
\end{itemize}
\end{corollary}

\noindent\textbf{Examples:} 
\begin{itemize}
\item[i)] Since $+$ is a function from $\IR\times\IR$ to $\IR$, it follows that ${^*}+$ is a function from ${^*}\IR\times {^*}\IR$ to ${^*}\IR$ with ${^*}r {^*}+ {^*}s = {^*}(r+s)$ for all $r,s\in\IR$. 
\item[ii)] Since the symplectic form $\omega$ on $\IH$ is a map from $\IH\times\IH$ to $\IR$, its *-image $^*\omega$ is a map from ${^*\IH}\times{^*\IH}$ to $^*\IR$ which agrees with $\omega$ on $\IH\times\IH\subset{^*\IH}\times{^*\IH}$. Analogous statements hold true for the inner product $\<\cdot,\cdot\>$  and the complex structure $J_0$ on $\IH$. 
\item[iii)] Since, for all $n\in\IN$ and all $k\in\IN$, we know that $\sum_{i=1}^k$ is a function from $(\IR^n)^k$ to $\IR^n$, it follows that, now even for all $n\in {^*}\IN$ and all $k\in{^*\IN}$, ${^*}\sum_{i=1}^k$ is a function from $({^*}\IR^n)^k$ to ${^*}\IR^n$ with ${^*}\sum_{i=1}^k {^*}r_i={^*}\sum_{i=1}^{k-1} r_i + r_k$ for all $k\in{^*\IN}$. 
\item[iv)] Since every sequence $s=(s_n)_{n\in\IN}$ of real numbers is a function from $\IN$ to $\IR$, it follows that its $*$-image ${^*}s$ is a function from ${^*}\IN$ to ${^*}\IR$ with ${^*}s_n= {^*}(s_n)$ for all $n\in\IN$. \\ 
\end{itemize}

We make the following \\

\noindent\textbf{Convention:} \emph{If no confusion is likely to arise, we make the convention to identify each standard set $b\in V_{<\infty}$ with ${^*}[b]\subset {^*}b\in W_{<\infty}$. In particular, we have $\IR\subset {^*}\IR$ and $\IN\subset {^*}\IN$.} \\

The saturation principle implies that the non-standard model $W=(W_n)_{n\in\IN}$ is (much) larger than the standard model $V=(V_n)_{n\in\IN}$. For the next statement we refer to (\cite{L2}, proposition 2.4.6) and (\cite{L2}, proposition 2.9.7).

\begin{proposition} We have the following dichotomy:
\begin{itemize}
\item[i)] If $b\in V_{<\infty}$ has finitely many elements, then its $*$-image ${^*}b\in W_{<\infty}$ consists of the $*$-images of its elements, $$^*\{a_1,\ldots,a_n\}=\{^*a_1,\ldots,^*a_n\}.$$
\item[ii)] If $b\in V_{<\infty}$ has infinitely many elements, then its $*$-image ${^*}b\in W_{<\infty}$ contains $b$ as a proper subset, $${^*}[b]=\{{^*}a: a\in b\}\subsetneq {^*}b.$$ 
\end{itemize}
In particular, it follows from ii) that $*: V_{<\infty}\to W_{<\infty}$ is a proper embedding.
\end{proposition}

\begin{proof}
While the part i) follows from the transfer principle after observing that the equality $b=\{a_1,\ldots,a_n\}$ can be incoded into the sentence $a\in b \Leftrightarrow a=a_1\vee\ldots\vee a=a_n$ in $\LL_V$, for part ii) consider the collection of sets $(a_i)_{i\in b}$ given by $a_i:={^*}b\backslash\{i\}$ for $i\in b$. While it easy to see that they have the finite intersection property, $a_{i_1}\cap\ldots\cap a_{i_n}\neq\emptyset$ for all $i_1,\cdots,i_n\in b$, $n\in\IN$, every element in $\bigcap_{i\in b} a_i\neq\emptyset$ is an element of ${^*}b\backslash b$. Note that, while in part i) the finite intersection property fails, in part ii) the transfer principle cannot be applied as the corresponding sentence would have infinite length, which is forbidden.   
\end{proof}

In particular, one can show that ${^*}\IR$, the set of *-real (or hyperreal or non-standard real) numbers, contains infinitesimals as well as numbers which are greater than any real number. 

\begin{proposition}\label{infinitesimal} The saturation principle implies the existence of the following ideal objects. 
\begin{itemize}
\item[i)] There exist $r\in{^*}\IR\backslash\{0\}$ such that $|r|<1/n$ for every standard natural number $n\in\IN$. Any such $r\in{^*}\IR$ (including $r=0$) is called \emph{infinitesimal} and we write $r\approx 0$.
\item[ii)] There exist $r\in{^*}\IR$ such that $|r|>n$ for every standard natural number $n\in\IN$. Any such $r\in{^*}\IR$ is called \emph{unlimited}. Any $r\in{^*}\IR$ which is not unlimited is called \emph{limited}. 
\item[iii)] A number $r\in{^*}\IR$ is limited if and only if it is \emph{near-standard} in the sense that there exists a standard real number $s\in\IR$ with $r-s\approx 0$. For every near-standard $r\in{^*}\IR$ we call ${^{\circ}}r:=s\in\IR$ the \emph{standard part} of $r$. 
\item[iv)] Every $n\in{^*\IN}\backslash\IN$ is unlimited. 
\end{itemize}
\end{proposition}

\begin{proof} For the definitions we refer to (\cite{L2}, definitions 1.2.7 and 1.6.9). Since the existence of infinitesimal and unlimited numbers is the key reason why to care about non-standard analysis, let us give the short proof: Define for every $n\in\IN$ the sets $a_n:=\{r\in{^*}\IR: 0<|r|<1/n\}$ and $b_n:=\{r\in{^*}\IR: |r|>n\}$. Since the corresponding collections of sets obviously have the finite intersection property, we find that $\bigcap_{n\in\IN} a_n$ and $\bigcap_{n\in\IN} b_n$ are non-empty and any element in these sets has the desired properties. For the third part we refer to (\cite{L2}, proposition 1.6.11). Part iv) follows from the observation that there are only finitely many natural numbers smaller than a given one, so the $*$-image of the corresponding set does not contain any new elements. \end{proof}

\begin{remark} \label{*-spheres} Along the same lines we have:
\begin{itemize}
\item[i)] Similar statements clearly hold when $^*\IR$ is replaced by $^*\IR^n$ for some standard $n\in\IN$. In particular, for every limited $r>0$ every point on ${^*S}^{n-1}(r)\subset{^*\IR}^n$ is near-standard, i.e., ${^*S}^{n-1}(r)$ is obtained from $S^{n-1}(r)$ by adding points which are infinitesimally close.
\item[ii)] In the same way as $^*\IR$ contains much more elements than $\IR$ itself, the non-standard extension $^*\IH$ of $\IH$ is a much larger space than $\IH$ itself.
\end{itemize}
\end{remark}   

In (\cite{L2}, theorems 1.6.8 and 1.6.15) it is shown that limited and infinitesimal numbers furthermore have the following nice closure properties. 

\begin{proposition}\label{limited-infinitesimal} We have
\begin{itemize}
\item[i)] Finite sums, differences and products of limited numbers are limited.
\item[ii)] Finite sums, differences and products of infinitesimal numbers are infinitesimal.
\item[iii)] The product of an infinitesimal number with a limited number is still infinitesimal.
\item[iv)] The standard part of a sum, difference or product of two limited numbers is the sum, difference or product of their standard parts.
\end{itemize}
\end{proposition}
 
\begin{remark} In an analogous way we prove in section $5$ that every infinite-dimensional (separable) Hilbert space $\IH$ is contained in a *-finite-dimensional Euclidean vector space $\IF$ of some unlimited but *-finite dimension $N\in{^*}\IN\backslash\IN$. The infinite-dimensional Hilbert space $\IH$ is \emph{not} a *-finite-dimensional Euclidean vector space itself, but is only \emph{contained} in some space which behaves as if it were finite-dimensional. \end{remark}

Apart from showing that the non-standard model contains infinitely-large numbers, the saturation principle immediately leads to the following, even more surprising fact, see (\cite{L2}, theorem 2.9.2).

\begin{proposition} For every standard set $b\in V_{<\infty}$ there exists a non-standard set $c\in W_{<\infty}$, which contains all elements of $a$, i.e., $a\in b$ implies ${^*}a\in c$, and which is \emph{*-finite} in the sense that there is a bijection from $c$ to an internal set $\{n\in{^*}\IN: n\leq N\}$ for some $N\in{^*}\IN$. \end{proposition}

Since a subset of a finite set in the standard model $V=(V_n)_{n\in\IN}$ is again a finite set, this seems to lead to an obvious logical contradiction. However, since the transfer principle only applies to subsets of *-finite sets which belong to the non-standard model, i.e., are internal themselves, the logical paradoxon is resolved in the following

\begin{proposition} For every infinite set $b\in V_{<\infty}$, the corresponding proper subset $b={^*}[b]=\{{^*}a: a\in b\}$ of ${^*}b$ is external. For example, $\IR$ and $\IN$ are external. In particular, the non-standard model is \emph{not full}, $W_n\subsetneq\PP(W_0\cup\ldots\cup W_{n-1})$. \end{proposition}
 
For the short proof we refer to (\cite{L2}, proposition 2.9.6). While these results are satisfactory from the theoretical point of view, for practical purposes it is rather important to know which subsets of an internal set in the non-standard model are still internal themselves, so that statements can be proven for them by applying the transfer principle. Since in applications one is almost exclusively interested in subsets which can be defined by requiring that their elements have a specific property, the following positive result originally due to Keisler, (\cite{Ke}, theorem 15.14), see also (\cite{L2}, theorem 2.8.4), is sufficient for all our purposes. 

\begin{proposition}\label{internal-definition} (Internal Definition Principle) Every \emph{definable set} belongs to the non-standard model $W=(W_n)_{n\in\IN}$ , that is, for every formula $A(x)$ in $\LL_W$ the set $\{a\in W_{<\infty}: W\models A(x)(a)\}$ is internal. \end{proposition}

In particular, every finite subset of an internal set is internal. \\

The fact that every internal set containing an infinite standard set as a subset must be strictly larger leads to the so-called \emph{spillover principles}, see (\cite{L2}, theorem 2.8.12). 

\begin{proposition}\label{spillover} Let $b$ denote an internal subset of ${^*}\IN$. Then it holds: 
\begin{itemize}
\item[i)] If for every $m\in\IN$ there exists some $n\geq m$ with $n\in b$, then $b$ must contain an unlimited *-natural number. 
\item[ii)] If for every unlimited *-natural number $N$ there exists a *-natural number $n\leq N$ with $n\in b$, then $b$ must also contain a standard natural number. 
\end{itemize}
\end{proposition}

\begin{proof} To prove i) define for every $m\in\IN$ the internal subset $b_m=\{n\in b: n\geq m\}$ of $b$. Since $b_{m_1}\cap\ldots\cap b_{m_k}\neq\emptyset$ for every finite collection $m_1,\ldots,m_k\in\IN$, it follows by the saturation principle that $\bigcap_{m\in\IN} b_m$ must contain an element which is an element of $b$ and greater than every standard natural number. In order to prove ii) observe that, by transfer, $b$ must have a minimal element $n$. If $n$ was unlimited, then there must exist $m\in b$ with $m\leq n-1$, contradicting the minimality of $n$. \end{proof}   

Finally, one of the main benefits of non-standard analysis is that the clumpsy $\epsilon$-formalism can be avoided by introducing infinitesimals and unlimited *-natural numbers. For the following proposition we refer to (\cite{L2}, theorem 1.7.1).

\begin{proposition}\label{convergence} 
A sequence $(s_n)_{n\in\IN}$ of real numbers converges to zero, $s_n\to 0$, as $n\to\infty$ if and only if $s_N:={^*s_N}\approx 0$ for all unlimited $N\in{^*}\IN\backslash\IN$. 
\end{proposition}

\begin{proof} First assume that $s_n\to 0$ as $n\to\infty$. By definition we know that for all $\eps>0$ there exists $n_0\in\IN$ such that $\forall n\in\IN: n\geq n_0\Rightarrow |s_n|<\eps$. By transfer, it follows that $\forall n\in{^*\IN}: n\geq n_0\Rightarrow |{^*s_n}|<\eps$. Since every unlimited $N\in{^*}\IN\backslash\IN$ is greater than every standard $n_0\in\IN$, it follows that $|{^*s_N}|<\eps$ for all standard $\eps>0$, that is, $|{^*s_N}|\approx 0$. In the opposite direction, assume that $|{^*s_N}|\approx 0$ for all unlimited $N\in{^*}\IN\backslash\IN$, in particular, for every standard $\eps>0$ and all unlimited $N$ there exists some $m\leq N$ such that $\forall n\in{^*\IN}: n\geq m\Rightarrow|{^*s_n}|<\eps$. By the spillover principle it follows that there must exist some standard $m\in\IN$ such that $\forall n\in\IN: n\geq m\Rightarrow|s_n|<\eps$, that is, $s_n\to 0$ as $n\to\infty$. \end{proof}

Since convergence in metric spaces is defined by requiring that the distance between points converges to zero, the above result immediately generalizes to all metric spaces.


\begin{thebibliography}{100000}
\bibitem{AM} Abbondandolo, A., Majer, P., \emph{A non-squeezing theorem for convex symplectic images of the Hilbert ball.} Arxiv preprint (1405.3200), 2014.
\bibitem{B} Berti, M., \emph{Nonlinear Oscillations of Hamiltonian PDEs.} Springer, 2007.
\bibitem{BEHWZ} Bourgeois, F., Eliashberg, Y., Hofer, H., Wysocki, K. and Zehnder, E., \emph{Compactness results in 
      symplectic field theory.} Geom. and Top. 7, pp. 799-888, 2003. 
\bibitem{DS} Dostoglou, S., Salamon, D., \emph{Self-dual instantons and holomorphic curves.} Ann. Math. 139, pp. 581-640, 1994.
\bibitem{EK} Eliasson, H., Kuksin, S., \emph{KAM for the nonlinear Schrödinger equation.} Ann. Math. 172, pp. 371-435, 2010.
\bibitem{Fa} Fabert, O., \emph{Infinite-dimensional symplectic non-squeezing using non-standard analysis.} Arxiv preprint (1501.05905), 2015.
\bibitem{Fl} Floer, A., \emph{Symplectic fixed points and holomorphic spheres.} Comm. Math. Phys. 120(4), pp. 575-611, 1989.
\bibitem{Fo} Fortune, B., \emph{A symplectic fixed point theorem for $\CP^n$.} Inv. Math. 81(1), pp. 29-46, 1985.
\bibitem{Gr} Gromov, M., \emph{Pseudo holomorphic curves in symplectic manifolds.} Inv. Math. 82(2), pp. 307-347, 1985. 
\bibitem{Ke} Keisler, H., \emph{Foundations of infinitesimal calculus.} Boston: Prindle, Weber and Schmidt, 1976. Online available at http://creativecommons.org/licenses/by-nc-sa/2.5/
\bibitem{K} Kuksin, S., \emph{Infinite-dimensional symplectic capacities and a squeezing theorem for Hamiltonian PDE’s.} Commun. Math. Phys. 167, pp. 531–552, 1995.
\bibitem{L1} Loeb, P., \emph{Simple Nonstandard Analysis and Applications.} in: Loeb, P., Wolff, M., Nonstandard Analysis for the Working Mathematician, Mathematics and its applications, Springer, 2015.
\bibitem{L2} Loeb, P., \emph{An Introduction to General Nonstandard Analysis.} in: Loeb, P., Wolff, M., Nonstandard Analysis for the Working Mathematician, Mathematics and its applications, Springer, 2015.
\bibitem{MDSa} McDuff, D. and D. Salamon, \emph{$J$-holomorphic curves and symplectic topology.} AMS Colloquium Publications 52, 2004.
\bibitem{Os} Osswald, H., \emph{The existence of polysaturated models.} in: Non-standard Analysis for the Working Mathematician,
 Mathematics and its applications, Kluwer, 2000.
\bibitem{OS} Osswald, H., \emph{Malliavin calculus in abstract Wiener space using infinitesimals.} Adv. Math. 176(1), pp. 1-37, 2003.
\bibitem{Ra} Rabinowitz, P., \emph{Free vibrations for a semi-linear wave equation.} Comm. Pure Appl. Math. 31(1), pp. 31-68, 1978.
\bibitem{R} Robinson, A., \emph{Non-standard analysis.} Princeton University Press, 1996.
\bibitem{RS} Robbin, J., Salamon, D., \emph{Asymptotics of holomorphic strips.} Ann. I.H. Poincare 18(5), pp. 573-612, 2001.
\bibitem{Sch} Schwarz, M., \emph{A quantum cup-length estimate for symplectic fixed points.} Inv. Math. 133(2), pp. 353-397, 1998.
\bibitem{ST} Sukhov, A., Tumanov, A., \emph{Symplectic non-squeezing in Hilbert space and discrete Schr\"odinger equations.} Arxiv preprint (1411.3989), 2014.
\end{thebibliography}
\end{document}